\documentclass[11pt]{amsart}
\usepackage{graphicx,amssymb,amsmath,amsthm}
\usepackage{enumerate}
\usepackage{dsfont}
\usepackage[colorlinks, citecolor=red]{hyperref}
\usepackage{comment,cite,color}
\usepackage{cite,color}
\usepackage{mathrsfs}
\usepackage{epsfig}
\usepackage{lscape}
\usepackage{subfigure}
\usepackage{epstopdf}
\usepackage{caption}
\usepackage{bm}
\usepackage{algorithm}
\usepackage{algpseudocode}
\usepackage[english]{babel}
\usepackage{multirow}

\newcommand{\xkh}[1]{\left(#1\right)}
\newcommand{\dkh}[1]{\left\{#1\right\}}
\newcommand{\zkh}[1]{\left[#1\right]}

\newcommand{\nj}[1]{\left\langle {#1} \right\rangle}

\newcommand{\norm}[1]{\left\|{#1}\right\|_2}
\newcommand{\normpsi}[1]{\left\|{#1}\right\|_{\psi_1}}

\newcommand{\norms}[1]{\left\|{#1}\right\|}
\newcommand{\abs}[1]{\lvert#1\rvert}
\newcommand{\Abs}[1]{\left\lvert#1\right\rvert}

\newcommand{\E}{{\mathbb E}}

\newcommand{\PP}{{\mathbb P}}

\newcommand{\N}{{\mathcal N}}
\newcommand{\1}{{\mathds 1}}

\newcommand{\R}{{\mathbb R}}

\newcommand{\T}{\top}
\newcommand{\C}{{\mathbb C}}

\newcommand{\vx}{{\bm x}}
\newcommand{\vxs}{{\bm x}^\sharp}

\newcommand{\vy}{{\bm y}}
\newcommand{\vu}{{\bm u}}
\newcommand{\vv}{{\bm v}}
\newcommand{\vw}{{\bm w}}
\newcommand{\vz}{{\bm z}}

\newcommand{\vzl}{{\bm z}^{(l)}}
\newcommand{\Al}{A^{(l)}}
\newcommand{\Fl}{F^{(l)}}
\newcommand{\Hl}{H^{(l)}}
\newcommand{\hH}{{\widehat{H}}}
\newcommand{\tA}{{\widetilde{A}}}
\newcommand{\tAl}{{\widetilde{A}}^{(l)}}
\newcommand{\hHl}{{\widehat{H}^{(l)}}}

\newcommand{\tz}{{\widetilde{\bm{z}}}}

\newcommand{\tzl}{{{\widetilde{\bm{z}}}}^{(l)}}

\newcommand{\vb}{{\bm b}}

\newcommand{\va}{{\bm a}}

\newcommand{\dist}{{\rm dist}}
\newcommand{\mutual}{{\rm mutual}}

\newcommand{\RNum}[1]{\uppercase\expandafter{\romannumeral #1\relax}}

\newtheorem{definition}{Definition}[section]

\newtheorem{theorem}[definition]{Theorem}
\newtheorem{lemma}[definition]{Lemma}

\newtheorem{remark}[definition]{Remark}

\newtheorem{example}[definition]{Example}

\date{}

\begin{document}

\author{Meng Huang}
\address{School of Mathematical Sciences, Beihang University, Beijing, 100191, China} \email{menghuang@buaa.edu.cn}
\thanks{The work of the author was supported by NSFC grant (12201022) and the Fundamental Research Funds for the Central Universities.}

\baselineskip 18pt
\bibliographystyle{plain}
\title[Asymptotic quadratic convergence of the Gauss-Newton method]{Asymptotic quadratic convergence of the Gauss-Newton method for complex phase retrieval}
\maketitle

\begin{abstract}
In this paper, we introduce a Gauss-Newton method for solving the complex phase retrieval problem. In contrast to the real-valued setting, the Gauss-Newton matrix for complex-valued signals is rank-deficient and, thus, non-invertible. To address this, we utilize a Gauss-Newton step that moves orthogonally to certain trivial directions. We establish that this modified Gauss-Newton step has a closed-form solution, which corresponds precisely to the minimal-norm solution of the associated least squares problem.  Additionally, using the leave-one-out technique, we demonstrate that $m\ge O( n\log^3 n)$ independent complex Gaussian random measurements ensures that the entire trajectory of the Gauss-Newton iterations remains confined within a specific region of incoherence and contraction with high probability. This finding allows us to establish the asymptotic quadratic convergence rate of the Gauss-Newton method without the need of sample splitting.
\end{abstract}
\keywords{Keywords: Complex phase retrieval, Minimal-norm Gauss-Newton, Leave-one-out, Quadratic convergence}

\section{Introduction}

\subsection{Problem setup}
Let $\vxs \in \C^n$  be an arbitrary unknown vector. The problem  of recovering $\vxs$ from systems of quadratic equations 
\begin{equation} \label{eq:problesetup}
y_j=\Abs{\nj{\va_j,\vxs}}^2, \quad j=1,\ldots,m,
\end{equation}
is termed as {\em phase retrieval}. Here, $\va_j \in \C^n$ are known sampling vectors and $y_j \in \R$ are observed measurements. Such problems are ubiquitous in many areas of  physical sciences and engineering, such as X-ray crystallography \cite{harrison1993phase,millane1990phase}, diffraction imaging \cite{shechtman2015phase,chai2010array}, microscopy \cite{miao2008extending}, astronomy \cite{fienup1987phase}, optics and acoustics \cite{walther1963question, balan2010signal,balan2006signal} etc, where the optical sensors and detectors are incapable of recording the phase information.
In addition to its applications in the physical sciences, solving systems of quadratic equations \eqref{eq:problesetup} is also crucial in the field of machine learning. This includes the training of neural networks that employ quadratic activation functions \cite{soltanolkotabi,du2018power}. 

%Beyond physical sciences, solving systems of quadratic equations \eqref{eq:problesetup} also has applications in machine learning, including the training of neural networks with quadratic activations  [44,45].

A natural approach for inverting the system of quadratic equations \eqref{eq:problesetup} is to solve the classical Wirtinger flow based model:
\begin{equation} \label{eq:mod1}
\min_{\vz \in \C^n} \quad f(\vz):=\frac1{m} \sum_{j=1}^m \xkh{\abs{\va_j^* \vz}^2- y_j }^2.
\end{equation}
It has been shown theoretically that $m\ge 4n-4$  generic measurements  suffice to guarantee that  the solution to \eqref{eq:mod1} is exact $\vxs$, up to a global phase \cite{conca2015algebraic,wangxu}.  In this paper, we employ the Gauss-Newton (GN) method to  solve \eqref{eq:mod1}. For the real-valued signals $\vxs$, the Gauss-Newton algorithm has been investigated in \cite{Gaoxu}. The authors demonstrate that that when $\va_j, j=1,\ldots,m$ are Gaussian random vectors, the Gauss-Newton method with resampling exhibits quadratic convergence with high probability, provided $m\ge O\xkh{\log\log(1/\epsilon) n\log n}$. Here, $\epsilon>0$ denotes the accuracy of the algorithm.  On one hand, the theoretical results presented in \cite{Gaoxu} require an infinite number of samples as $\epsilon\to 0$. On the other hand, the Gauss-Newton method with resampling necessitates partitioning the sampling vectors  $\va_j$  into a series of disjoint blocks of roughly equal size. This approach is impractical due to the lack of information about the number of blocks or their appropriate sizes. Additionally, when dealing with complex-valued signals $\vxs$, the Gauss-Newton matrix is rank-deficient and thus singular. To address this issue, the authors in \cite{Gaoxu} suggest using the minimal-norm Gauss-Newton method to solve \eqref{eq:mod1}.  The effectiveness of this approach has been verified through numerical experiments; however, there is no theoretical guarantee for its success.  Another way to address this issue is by utilizing the Levenberg-Marquardt (LM) method, an algorithm that resembles the Gauss-Newton method but adds a regularization term to the GN matrix. The authors in \cite{Mawen} demonstrate that, under the complex Gaussian setting, the LM method exhibits linear convergence for complex phase retrieval with high probability, provided that $m \ge O(n \log n)$.

To date, as far as we know, no algorithm exists that achieves a provable quadratic convergence rate for solving the phase retrieval problem without requiring sample splitting. Motivated by this, we aim to give a theoretical understanding about the convergence properties of the Gauss-Newton method for complex setting, and we are interested in the following question: {\em Can we establish the quadratic convergence rate of the Gauss-Newton method for phase retrieval with no need of sample splitting, especially for the complex-valued signals?}

\subsection{Related work}
 The phase retrieval problem, which seeks to recover $\vx$ from systems of quadratic equations \eqref{eq:problesetup}, has undergone intensive investigation in recent years. Over the past two decades, numerous algorithms with provable performance guarantees have been developed to address this problem. One prominent line of research involves convex relaxation, which first transforms the phase retrieval problem into a low-rank matrix recovery problem, followed by relaxation to a nuclear norm minimization problem. Theoretical analyses have demonstrated that the convex relaxation approach is effective provided the sampling complexity $m \ge Cn$, where $C$ is a sufficiently large constant \cite{phaselift,Phaseliftn,Waldspurger2015,Chenchi, kueng, Tcai}. However, the resultant semidefinite program tends to be computationally inefficient for large-scale problems. To mitigate this issue, another research trajectory optimizes a non-convex loss function within the natural parameter space. Netrapalli et al. proved that the alternating minimization method with resampling, based on spectral initialization, can achieve $\epsilon$ accuracy from $O(n \log n (\log^2 n+\log1/\epsilon))$ Gaussian random measurements \cite{AltMin}. Subsequently, Candès et al. demonstrated that the Wirtinger flow algorithm with spectral initialization achieves linear convergence with $O(n \log n)$ Gaussian random measurements, marking the first convergence guarantee for non-convex methods without resampling \cite{WF}. Chen and Candès further improved this result to $O(n)$ Gaussian random measurements by incorporating an adaptive truncation strategy \cite{TWF}. Ma et al. revisited the vanilla gradient descent method, utilizing leave-one-out arguments, and established local linear convergence with enhanced computational efficiency \cite{macong}. Additional non-convex phase retrieval algorithms with provable guarantees are discussed in \cite{bendory2017non,waldspurger2018phase,turstregion,cai2021solving,TAF,RWF,ChiYm,TTCai,caiLiu,ChenFan,Duchi,tan2019phase,huangsiam,ZYL}. For a comprehensive overview of recent theoretical, algorithmic, and application developments in phase retrieval, readers are referred to survey papers \cite{jaganathan2016phase,shechtman2015phase}. It is important to note that almost all algorithms for phase retrieval exhibit only linear convergence, with the notable exception of the Gauss-Newton method proposed in \cite{Gaoxu}, which demonstrates quadratic convergence. However, the findings in \cite{Gaoxu} not only necessitate sample splitting but also do not accommodate scenarios involving complex-valued signals.

Our analysis utilizes the leave-one-out argument, initially proposed for analyzing high-dimensional convex problems with random designs. This approach has been applied to a range of applications, including M-estimation \cite{Eikaroui1,Eikaroui2}, phase synchronization \cite{Zhong,Abbe}, and the maximum likelihood ratio test for logistic regression \cite{sur}, among others. Ma et al. adopted this method to examine non-convex optimization algorithms, successfully establishing the linear convergence of gradient descent for phase retrieval, matrix completion, and blind deconvolution \cite{macong}. Chen et al. utilized this argument to provide theoretical guarantees for the gradient descent of phase retrieval with random initialization \cite{ChenFan}. Additional applications of leave-one-out arguments that enhance computational efficiency and improve performance bounds for non-convex algorithms including \cite{macong,LiY,ChenFW,ChenNoisy,ChenJ}.

\subsection{Our Contributions}
To date, the Gauss-Newton method remains the sole phase retrieval algorithm with provable quadratic convergence. For real-valued signals, the theoretical guarantees are established in \cite{Gaoxu} with sample splitting.  For complex-valued signals, due to the singularity of the Gauss-Newton matrix, a Levenberg-Marquardt (LM) method, which resembles the Gauss-Newton method, is proposed for phase retrieval in \cite{Mawen}. However, this method only achieves linear convergence. Currently, there is no documented result concerning the quadratic convergence of the Gauss-Newton method without sample splitting. The goal of this paper is to provide the theoretical guarantees for it.  

To address the singularity of the Gauss-Newton matrix, inspired by the modified trust-region method \cite{turstregion}, we constrain each Gauss-Newton step $\delta_k$ to be geometrically orthogonal to the trivial direction $i\vz_k$. In this context, $\C^n$ is treated as $\R^{2n}$, and two complex vectors $\vz, \vw \in \C^n$ are considered orthogonal if $\Re(\vw^* \vz)=0$. Utilizing the properties of the Moore-Penrose pseudoinverse, we demonstrate that our modified algorithm corresponds precisely with the minimal-norm Gauss-Newton method. This method has been proven to exhibit semi-local convergence for general nonlinear least squares under center-Lipschitz conditions \cite{Haussler,Argyros,Pes}, however,  its application in phase retrieval remains unexplored.  Utilizing leave-one-out arguments, we establish that, with high probability, the minimal-norm Gauss-Newton method for phase retrieval achieves asymptotic quadratic convergence without requiring sample splitting, provided the number of measurements $m \ge C_0 n \log^3 m$ for some sufficiently large constant $C_0 > 0$.

We emphasize that, although the leave-one-out argument is commonly used for analyzing first-order algorithms, our paper is the first to apply it to a second-order algorithm. This is notably more challenging, as the leave-one-out argument for second-order algorithms typically involves perturbations of the inverse of a matrix.

\subsection{Notations}
\subsubsection{Basic notations}
Throughout the paper, we assume the $\va_j \in \C^n, j =1,\ldots, m$ are independent identically distributed (i.i.d.) complex Gaussian random vector, namely,  $\va_j \sim 1/\sqrt2\cdot \N(0,I_n)+i /\sqrt2\cdot \N(0,I_n)$.  We use $ \mathbb{S}_{\C}^{n-1}$ for the complex unit sphere in $\C^n$.  Let $\Re(\vz) \in \R^n$ and $\Im(\vz) \in \R^n$ denote the real and imaginary part of a complex vector $\vz\in \C^n$. We will often use the canonical identification of $\C^n$ and $\R^{2n}$, which assign $\vz \in \C^n$ to $[\Re(\vz); \Im(\vz)] \in \R^{2n}$. For this reason, we say two complex vectors $\vz, \vw \in \C^n$ are orthogonal if and only if $\Re(\vw^* \vz)=0$.  For a vector $\vz \in \C^n$, we use $\vz(k:l)$ to denote the vector consisting of entries from the $k^{th}$ to the $l^{th}$ position, where $1 \le k \le l \le n$.
 The notation $f(n)=O(g(n))$ or $f(n) \lesssim g(n)$ (resp. $f(n) \lesssim g(n)$) means there exists a constant $c_0>0$ such that $f(n) \le c_0 g(n)$ (resp. $f(n) \ge c_0 g(n)$). 
  For any matirx $A\in \C^{m\times n}$, we use $\norm{A}$ and $\norms{A}_F$ to denote its the spectral norm and the Frobenius norm, respectively. Moreover, we define 
\[
\norms{A}_{2,\Re} = \max_{\vx \in \R^n} \frac{\norm{A \vx}}{ \norm{\vx}}. 
\]
It is easy to check $\norms{A}_{2,\Re}=\norm{[\Re(A); \Im(A)]}$, where $\Re(A) \in \R^{m\times n} $ and $\Im(A)  \in \R^{m\times n} $ the real and imaginary part of the matrix $A$, respectively.

Obviously, for any $\vz$ if  $\vz$ is a solution to \eqref{eq:mod1} then $\vz e^{i\phi}$ is also a solution to it for any $\phi\in \R$. For this reason, we define the distance between $\vz$ and $\vxs$ as
\[
\dist(\vz,\vx)=\min_{\phi\in \R}\norm{\vz-\vxs e^{i\phi}}.
\]
For convenience, we also define the phase $\phi(\vz)$ as
\begin{equation} \label{eq:defphi}
\phi(\vz):=\mbox{argmin}_{\phi \in \R} \norm{\vz-\vxs e^{i\phi}}
\end{equation}
for any $\vz \in \C^n$. It is easy to verify that $\Im(\vz^* \vxs e^{i\phi(\vz)})=0$. 

\subsubsection{Wirtinger calculus}
Consider a real-valued function $f: \C^d\to \R$. According to the  Cauchy-Riemann conditions, $f$ is not complex differentiable unless it is constant.
However, when considering $f(\vz)$ as a function of  $(\vx,\vy)\in \R^{d}\times \R^d \cong \C^d$ where $\vx:=\Re(\vz), \vy:=\Im(\vz)$, it becomes feasible for $f(\vx,\vy)$ to be differentiable in the real sense. Direct differentiation of $f$ with respect to $\vx$ and $\vy$ can be complex and cumbersome. A more streamlined method involves using Wirtinger calculus, which simplifies the derivative expressions significantly, making them resemble to those with respect to $\vx$ and $\vy$ directly. Here, we provide a concise exposition of Wirtinger calculus (see also \cite{WF,turstregion}).

For any real-valued function $f(\vz)$, we can write it in the form of $f(\vz,\bar{\vz})$, where $\vz=\vx+i\vy$ and $\bar{\vz}:=\vx-i\vy$. Here, $\vx:=\Re(\vz), \vy:=\Im(\vz)$.
 If $f$ is differentiable as a function of $(\vx,\vy)\in \R^d\times \R^d$ then the Wirtinger gradient  is well-defined and can be denoted by
 \[
\nabla f(\vz)=\zkh{\frac{\partial f}{\partial \vz}, \frac{\partial f}{\partial \bar{\vz}}}^*,
\]
where
\[
\frac{\partial f}{\partial \vz}:=\frac{\partial f(\vz,\bar{\vz})}{\partial \vz} \Bigg|_{\bar{\vz}=\mbox{constant}}=\zkh{\frac{\partial f(\vz,\bar{\vz})}{\partial z_1 } ,\ldots, \frac{\partial f(\vz,\bar{\vz})}{\partial z_d }  }\Bigg|_{\bar{\vz}=\mbox{constant}}
\]
and
\[
\frac{\partial f}{\partial \bar{\vz}}:=\frac{\partial f(\vz,\bar{\vz})}{\partial \bar{\vz}} \Bigg|_{\vz=\mbox{constant}}=\zkh{\frac{\partial f(\vz,\bar{\vz})}{\partial \bar{z}_1 } ,\ldots, \frac{\partial f(\vz,\bar{\vz})}{\partial \bar{z}_d }  }\Bigg|_{\vz=\mbox{constant}}.
\]
Here, when applying the operator $\frac{\partial f}{\partial \vz}$, $\bar{\vz}$ is formally treated as a constant, and similar to the operator $\frac{\partial f}{\partial \bar{\vz}}$.  
%The Hessian matrix in Wirtinger calculus  is defined as
%\[
%\nabla^2 f(\vz):=\left[ \begin{array}{ll}
%                           \frac{\partial }{\partial \vz} \xkh{\frac{\partial f}{\partial \vz} }^*   &   \frac{\partial }{\partial \bar{\vz}} \xkh{\frac{\partial f}{\partial \vz} }^* \vspace{1em}\\
%                            \frac{\partial }{\partial \vz} \xkh{\frac{\partial f}{\partial \bar{\vz}} }^* &  \frac{\partial }{\partial \bar{\vz}} \xkh{\frac{\partial f}{\partial \bar{\vz}} }^*
%                            \end{array} \right ].
%\]

\subsection{Moore-Penrose pseudoinverse}
Throughout the paper, we assume $A^{\dag}$ is the Moore-Penrose pseudoinverse of the matrix $A\in \C^{m\times n}$, which is  defined by means of the four “Moore-Penrose equations”
\[
AA^{\dag}A=A,\qquad A^{\dag}AA^{\dag}=A^{\dag},\qquad  (AA^{\dag})^*=AA^{\dag},\qquad (A^{\dag}A)^*=A^{\dag}A.
\]
Furthermore, if $A^* A$ is invertible then $A^\dag=(A^* A)^{-1}A^*$ and $\norm{A^\dag}^2=\norm{(A^* A)^{-1}}$.
For the definition and a comprehensive analysis of the properties of the Moore-Penrose inverse we refer the reader to \cite{Groetsch}.

\subsection{Organization}
The structure of this paper is as follows: In Section 2, we introduce a modified Gauss-Newton method for complex phase retrieval and establish that it is, in essence, the minimal-norm Gauss-Newton algorithm for general nonlinear least squares. Section 3 presents the main result of this paper, demonstrating that the minimal-norm Gauss-Newton method for phase retrieval achieves an asymptotic quadratic convergence rate. In Section 4, we evaluate the empirical performance of our algorithm through a series of numerical experiments. Section 5 provides an outline of the proof for the main result. Section 6 offers a brief discussion on potential future work. Appendices A and B contain the technical lemmas necessary for our analysis and the detailed proofs of technical results, respectively.

\section{The Gauss-Newton Method}
The program we consider is
\[
\min_{\vz \in \C^n} \quad f(\vz):=\frac1{m} \sum_{j=1}^m \xkh{\abs{\va_j^* \vz}^2- y_j }^2:=  \frac1{m} \sum_{j=1}^m \xkh{F_j(\vz)}^2,
\]
where $y_j=\abs{\va_j^* \vxs}^2$.  To solve this nonlinear least squares problem, we apply the well-known Gauss-Newton method. Under the Wirtinger calculus, the linearization of $F_j(\vz)$ at the point $\vz_k$ is
\[
F_j(\vz) \approx F_j(\vz_k) +  (\nabla F_j(\vz_k))^*  \left [ \begin{array}{l} \vz-\vz_k  \vspace{0.5em}\\ \overline{ \vz-\vz_k } \end{array}\right]. 
\]
Here, $\nabla F_j(\vz_k)$ is the Wirtinger gradient of $F_j(\vz)$ at the point $\vz_k$. The Gauss-Newton update rule is $\vz_{k+1}=\vz_k+\delta_k$, where $\delta_k$ is determined by solving the following linear least squares:
\begin{equation} \label{eq:minpro}
\min_{\delta \in \C^n} \quad \norm{ A(\vz_k) \left [ \begin{array}{l} \delta \vspace{0.5em}\\ \overline{ \delta } \end{array}\right]  + F(\vz_k) }^2,
\end{equation}
where
\begin{equation} \label{eq:Azk}
A(\vz_k):=\frac1{\sqrt{m}}   \left [ \begin{array}{cc} \vz_k^*\va_1\va_1^*,&  \vz_k^\T \bar{\va}_1\va_1^\T \\
\vdots & \vdots\\
\vz_k^*\va_m\va_m^*,&  \vz_k^\T \bar{\va}_m\va_m^\T \end{array} \right ] \in \C^{m\times 2n}
\end{equation}
and 
\begin{equation} \label{eq:Fzk}
 F(\vz_k)=\frac1{\sqrt{m}}  (\abs{\va_1^* \vz}^2- y_1,\ldots, \abs{\va_m^* \vz}^2- y_m  )^\T.
\end{equation}
Observe that $A(\vz_k) \left [ \begin{array}{l} \vz_k \vspace{0.5em}\\ -\overline{ \vz_k } \end{array}\right]=0$. This implies the matrix $A(\vz_k)$ is rank deficient, rendering the solution to \eqref{eq:minpro} non-unique. 
An  intriguing characteristic of $f(\vz)$ is that each point has a circle of equivalent points, all sharing the same function value. Therefore, to minimize the function value as effectively as possible, we constrain each update step to move in a direction orthogonal to the tangent vector $i\vz_k$ at point $\vz_k$. Specifically, we select  $\delta_k$ as the solution to
\begin{equation} \label{eq:minpro1}
\min_{\delta \in \C^n} \quad \norm{ A(\vz_k)  \left [ \begin{array}{l} \delta \vspace{0.5em}\\ \overline{ \delta } \end{array}\right]   + F(\vz_k) }^2, \qquad \mbox{s.t.} \qquad \Im(\delta^* \vz_k)=0.
\end{equation}
Notice that if we treat $\C^n$ as $\R^{2n}$ then $S(\vz_k):=\dkh{\vw\in \C^n:  \Im(\vw^* \vz_k)=0}$ forms a subspace of dimension $2n-1$ over $\R^{2n}$.  Take any matrix $U(\vz_k) \in \C^{n\times (2n-1)}$ whose columns forms an orthonormal basis for the subspace, i.e., $\Re(U_k^* U_l)=\delta_{k,l}$ for any columns $U_k$ and $U_l$. Then the problem \eqref{eq:minpro1} can be reformulated as 
\begin{equation} \label{eq:xik13}
\mbox{min}_{\xi \in \R^{2n-1}} \quad \norm{A(\vz_k) \left [ \begin{array}{l} U(\vz_k) \vspace{0.4em} \\ \overline{ U(\vz_k) } \end{array}\right]  \xi + F(\vz_k)  }^2.
\end{equation}
And $\xi_k$ is the solution to \eqref{eq:xik13} if and only if $\delta_k=U(\vz_k) \xi_k$ the solution to \eqref{eq:minpro1}.  

The following lemma shows that within the subspace $S(\vz_k):=\dkh{\vw\in \C^n:  \Im(\vw^* \vz_k)=0}$, the matrix $A(\vz_k)$ is invertible with high probability for all $\vz_k$ obeying $\max_{1\le j\le m} ~\abs{\va_j^* \vz_k} \le C_1 \sqrt{\log m} \norm{\vz_k}$. Here, $C_1>0$ is a universal constant.
 Consequently, the program \eqref{eq:minpro1} has a unique solution for all $\vz$ satisfying  the above ``near-independence'' property.  To facilitate our discussion, we define
\begin{equation} \label{eq:Hz}
H(\vz)=\left [ \begin{array}{l} U(\vz) \vspace{0.4em} \\ \overline{ U(\vz) } \end{array}\right]^* A(\vz)^* A(\vz) \left [ \begin{array}{l} U(\vz) \vspace{0.4em} \\ \overline{ U(\vz) } \end{array}\right] \in \R^{(2n-1)\times (2n-1)},
\end{equation}
where $U(\vz) \in \C^{n\times (2n-1)}$ represents the orthonormal matrix corresponding to the subspace $S(\vz):=\dkh{\vw\in \C^n:  \Im(\vw^* \vz)=0}$.

\begin{lemma} \label{le:Hzlowup}
Suppose that  the sample complexity obeys $m\ge C_0n\log^3 n$ for some sufficiently large constant $C_0>0$.  With probability at least $1-O(m^{-10})$
\[
 H(\vz) \succeq 1.9 \norm{\vz}^2 I_{2n-1}  \quad \mbox{and} \quad  A(\vz)^* A(\vz) \preceq 5\norm{\vz}^2 I_{2n} 
\]
holds simultaneously for all $\vz \in \C^n$ obeying 
\begin{equation*} 
\max_{1\le j\le m} ~\abs{\va_j^* \vz} \le C_1 \sqrt{\log m} \norm{\vz}.
\end{equation*}
Here, $C_1>0$ is a universal constant.
\end{lemma}
\begin{proof}
See Section \ref{sec:Hzlowup}.
\end{proof}

The subsequent lemma demonstrates that the unique solution to \eqref{eq:minpro1}  is precisely $-A(\vz_k)^{\dag} F(\vz_k)$, where $A(\vz_k)^{\dag}$ represents the Moore-Penrose pseudoinverse of $A(\vz_k)$, as illustrated below.

\begin{lemma} \label{le:psu}
For any $\vz_k \in \C^n$, the solution to  \eqref{eq:minpro1} is 
\[
\delta_k=-A(\vz_k)^{\dag} F(\vz_k)(1:n).
\]
Here,  $A(\vz_k)^{\dag} F(\vz_k)(1:n)$ denotes the vector consisting of entries from the first to the $n^{th}$ position of $A(\vz_k)^{\dag} F(\vz_k)$.
\end{lemma}
\begin{proof}
As demonstrated in \cite[Proposition III.1]{Gaoxu},  $-A(\vz_k)^{\dag} F(\vz_k)(1:n)$ is a solution to  the unconstrained program  \eqref{eq:minpro}. 
Therefore, to prove this lemma, it suffices to show that $ \Im(\vz_k^*A(\vz_k)^{\dag} F(\vz_k)(1:n))=0$. To this end, observe that
\begin{eqnarray*}
2 \Im(\vz_k^*A(\vz_k)^{\dag} F(\vz_k)(1:n)) &=&  \nj{ \left ( \begin{array}{l} i \vz_k \vspace{0.5em}\\ -i \overline{ \vz_k } \end{array}\right),  \left ( \begin{array}{l} A(\vz_k)^{\dag} F(\vz_k)(1:n) \vspace{0.5em}\\  \overline{ A(\vz_k)^{\dag} F(\vz_k)(1:n)}  \end{array}\right)} \\
& \overset{\text{(i)}}{=}&  \nj{\xkh{ I - A(\vz_k)^{\dag} A(\vz_k)}  \left ( \begin{array}{l} i \vz_k \vspace{0.5em}\\ -i \overline{ \vz_k } \end{array}\right), A(\vz_k)^{\dag} F(\vz_k)} \\
&\overset{\text{(ii)}}{=}&  \left ( \begin{array}{l} i \vz_k \vspace{0.5em}\\ -i \overline{ \vz_k } \end{array}\right)^* \xkh{ I - A(\vz_k)^{\dag} A(\vz_k)} A(\vz_k)^{\dag} F(\vz_k) \\
&\overset{\text{(iii)}}{=}& 0.
\end{eqnarray*}
Here, (i) arises from the fact that $A(\vz_k)^{\dag} F(\vz_k)(n+1:2n)=\overline{ A(\vz_k)^{\dag} F(\vz_k)(1:n)}$  \cite[Eq. (III.15)]{Gaoxu} and the identity 
\[
A(\vz_k) \left ( \begin{array}{l} i \vz_k \vspace{0.5em}\\ -i \overline{ \vz_k } \end{array}\right)=0,
\]
and (ii), (iii) come from the definition of Moore-Penrose pseudoinverse that $(A(\vz_k)^{\dag} A(\vz_k))^*=A(\vz_k)^{\dag} A(\vz_k)$ and $A(\vz_k)^{\dag} A(\vz_k) A(\vz_k)^{\dag} =  A(\vz_k)^{\dag}$. This concludes the proof.
\end{proof}

Equipped with Lemma \ref{le:psu}, our minimal-norm Gauss-Newton method for phase retrieval is then a combination of spectral initialization and the Gauss-Newton step, as detailed in in Algorithm \ref{al:1}.
To improve the efficiency of the Gauss-Newton method in practice, $A(\vz_k)^{\dag} F(\vz_k)$ can be solved inexactly after reaching certain criterion. More specifically, according to the properties of the Moore-Penrose pseudoinverse, $A(\vz_k)^{\dag} F(\vz_k)$ is the solution to 
\[
\min_{\bm{d}\in \C^{2n}} \quad \norm{\bm{d}} \quad  \mbox{s.t.}\quad   A^*(\vz_k)  A(\vz_k) \bm{d}= A^*(\vz_k)  F(\vz_k) .
\]
Several algorithms can solve it efficiently, such as LSQR \cite{paige} and CRAIG \cite{saunders}. In our numerical experiments, we solve it by LSQR and limit its maximum iteration number to be $10$.

\begin{algorithm}[H]
\caption{Minimal-norm Gauss-Newton for complex phase retrieval}
\label{al:1}
\begin{algorithmic}[H]
\Require
Measurement vectors: $\va_j \in \C^n, j=1,\ldots,m $; Observations: $y_j \in \C,  j=1,\ldots,m$; the maximum number of iterations $T$.   \\
\textbf{Spectral initialization:} Let $\lambda_1(Y)$ and $\tz_0 \in \C^n$ be the leading eigenvalue and eigenvector of 
\[
Y=\frac1m \sum_{j=1}^m y_j\va_j\va_j^*,
\]
respectively. And set $\vz_0=\sqrt{\lambda_1(Y)/2}\cdot \tz_0$. \\
\textbf{Gauss-Newton updates:}  for $k=0,1,\ldots, T-1$ 
\begin{equation} \label{eq:al1zk}
\vz_{k+1}=\vz_k- A(\vz_k)^{\dag} F(\vz_k)(1:n),
\end{equation}
where $A(\vz_k)$ and $F(\vz_k)$ are given in \eqref{eq:Azk} and \eqref{eq:Fzk}, respectively.
\Ensure
The vector $ \vz_T $.
\end{algorithmic}
\end{algorithm}

\section{Main results}

In this section, we demonstrate that the minimal-norm Gauss-Newton method for phase retrieval achieves asymptotic quadratic convergence rate. 
To this end,  we require certain Lipschitz conditions. Specifically,  the matrix $A(z)$, defined in \eqref{eq:Azk}, should exhibit local Lipschitz continuity such that $\norm{A(\vx) - A(\vy)} \le K\norm{\vx-\vy}$ for some constant $K > 0$. However,  the heavy-tailed behavior of the fourth powers of Gaussian random variables can cause the constant $K$ to reach values as large as $O(\sqrt{n})$. This necessitates that the initial vector $\vz_0 \in \C^n$ must satisfy $\mbox{dist}(\vz_0, \vxs) \lesssim 1/\sqrt{n}$, which is a rather stringent requirement. To address this, \cite{Gaoxu} employs sample splitting. However, this approach is not sample-efficient and is not typically implemented in practice.

Instead of sample splitting, we introduce the region 
\[
\mathcal{R}_0:=\dkh{\vz\in \C^n:  \dist(\vz_{0}, \vxs) \le \delta \|\vxs\|_2 \quad \mbox{and} \quad 
\max_{1\le j\le m} ~\abs{\va_j^* (\vz- \vxs e^{i\phi(\vz)})} \lesssim  \sqrt{\log m} },
\]
where $0<\delta<1$ is a universal constant. For any $\vz \in  \mathcal{R}_0$, $\vz$ is close to $\vxs$ and exhibits ``near-independence'' to all sensing vectors $\va_j$.  We term $\mathcal{R}_0$ the {\em Region of Incoherence and Contraction} (RIC).  The RIC has the advantageous property that the matrix $A(z)$, defined in \eqref{eq:Azk}, exhibits near-Lipschitz continuity for all $\vx,\vy \in \mathcal{R}_0$ as follows (see the Lemma \ref{le:conun1}):
\[
\norm{A(\vx) - A(\vy)} \le 4\norm{\vx-\vy} + \epsilon \qquad \mbox{for all} \quad \vx,\vy \in \mathcal{R}_0,
\]
where $\epsilon>0$ is a sufficiently small constant.   This will allows us to establish asymptotic quadratic convergence of the Gauss-Newton method.
A key question remains: how can we ensure that the trajectory of the Gauss-Newton update rule never leaves the RIC? To address this, we employ leave-one-out arguments to construct auxiliary sequences that use all but one sample.  This approach leverages the statistical independence between the auxiliary sequences and the corresponding sensing vector that has been left out, thereby establishing the "near-independence" property. 
By utilizing this technique, we have the result:

\begin{theorem} \label{th:main}
Let $\vxs \in \C^n$ be a fixed vector.  Suppose that $\va_j \in \C^n, j=1,\ldots,m,$ are i.i.d. complex Gaussian random vectors. For any sufficiently small constant $\epsilon_0>0$,  with probability at least $1- O(m^{-8})-O(me^{-1.5n})$, the initial estimate $\vz_0$ given in Algorithm \ref{al:1}  obeys
\begin{equation} \label{eq:main0}
 \dist(\vz_{0}, \vxs) \le \delta \|\vxs\|_2,
\end{equation}
and  the iterates $\vz_k$ given in \eqref{eq:al1zk} obey
\begin{subequations} 
\label{eq:main11}
\begin{gather} 
\dist(\vz_{k+1}, \vxs) \le \frac2{\norm{\vxs}} \dist^2(\vz_{k}, \vxs) +\epsilon_0 \dist(\vz_{k}, \vxs), \label{eq:main1} \\
\max_{1\le j\le m} ~\abs{\va_j^* (\vz_k- \vxs e^{i\phi(\vz_k)})} \le C_1 \sqrt{\log m} \|\vxs\|_2 \label{eq:main2},    
\end{gather}
\end{subequations}
for all $k\ge 0$, provided  $m\ge C_0 \epsilon_0^{-4} n\log^3 n$  for some large constant $C_0 >0$.  Here, $0< \delta \le  0.01$ is a universal constant.
\end{theorem}

\begin{remark}
The precise expression for the constant $\epsilon_0$ in Theorem \ref{th:main} is $\epsilon_0=c_1 \sqrt[4]{\frac{n\log^3m}{m}}$ for some universal constant $c_1>0$. Consequently, $\epsilon_0 \to 0$ as $m\to \infty$, leading to \eqref{eq:main1} becoming $\dist(\vz_{k+1}, \vxs) \le \frac2{\norm{\vxs}}  \dist^2(\vz_{k}, \vxs)$. Therefore,  our Gauss-Newton method exhibits an asymptotic quadratic convergence rate.
\end{remark}

\begin{remark}
Theorem \ref{th:main} demonstrates that Algorithm \ref{al:1} succeeds with high probability when the sampling complexity satisfies $m\gtrsim n\log^3 n$.  Notably, the minimal sample size should obey the condition  $m\ge \xkh{4+o(1)}n$ \cite{conca2015algebraic,Heinosaari}. Thus, our sample complexity is near-optimal up to a logarithmic factor.
\end{remark}

\begin{remark}
The combination of \eqref{eq:main0} and \eqref{eq:main11} implies that the trajectory of the Gauss-Newton update rule always remains within the region of incoherence and contraction. This demonstrates the weak statistical dependency between the iterates $\dkh{\vz_k}$ and the design vectors $\dkh{\va_j}$.
\end{remark}

\section{Numerical Experiments}
In this section, we present several numerical experiments to validate the effectiveness and robustness of the Gauss-Newton method in comparison with WF \cite{WF}, TWF \cite{TWF}, and TAF \cite{TAF}. These methods are selected due to their widespread applications and high efficiency in solving the phase retrieval problem. We employ the LSQR method to compute the Moore-Penrose pseudoinverse in Algorithm \ref{al:1}, limiting its maximum number of iterations to 10 for signal recovery and 5 for natural image recovery. All experiments are conducted on a laptop with a 2.4 GHz Intel Core i7 Processor, 8 GB 2133 MHz LPDDR3 memory, and Matlab R2024a.

\subsection{Recovery of signals with noiseless measurements}
During each experiment, the target vector  $\vxs \in \C^n$ is chosen randomly from the standard complex Gaussian distribution.  The measurement vectors $ \va_j, \,j=1,\ldots,m$ are generated either from the standard complex Gaussian distribution or according to the coded diffraction patterns (CDP) model. For the CDP model, we utilize octanary pattern masks as described in \cite{WF}. The implementations for WF, TWF, and TAF were obtained from the authors' websites, using the recommended parameters.

\begin{example}
We first assess the relative error of the Gauss-Newton method in comparison to WF, TWF, and TAF for both complex Gaussian and CDP cases. 
 We set $n=1000$.  For the complex Gaussian case, we choose $m=5n$.  For the  CDP case, we select the number of masks as $L=6$.   The results are presented in Figure \ref{figure:relative_error}. It is evident that the Gauss-Newton method converges significantly faster than the state-of-the-art algorithms, requiring only about  $10$ iterations to achieve a relative error of $10^{-15}$.
\end{example}

\begin{example}
In this example, we evaluate the empirical success rate of the Gauss-Newton method with respect to the number of measurements.  We set $n=1000$. For the complex Gaussian case,  we vary $m$ within the range $[2n,6n]$.  For the  CDP case,  we set the number of masks $m/n=L$  from $2$ to $6$.  For each $m$, we run $100$ times trials to calculate the success rate.  A trial is considered successful if the algorithm returns a vector $\vz_T$  with a small relative error, specifically when  $\mbox{dist}(\vz_{T}-\vx)/\norm{\vx} \le 10^{-5}$.  The results are plotted in \ref{figure:suc}.  It can be observed that the empirical success rate of the Gauss-Newton method is comparable to that of TAF and TWF, and even slightly better than WF.
\end{example}

\begin{figure}[H]
\centering
\subfigure[]{
     \includegraphics[width=0.45\textwidth]{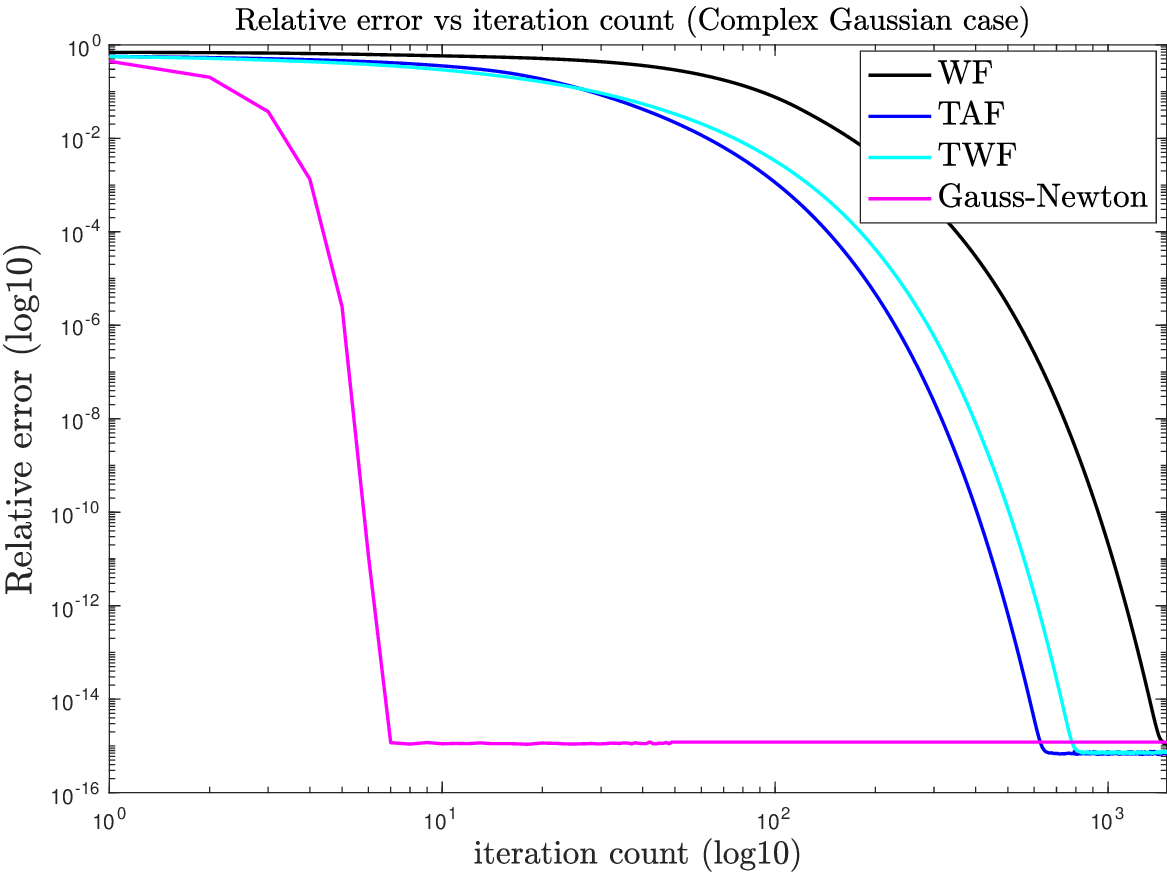}}
\subfigure[]{
     \includegraphics[width=0.45\textwidth]{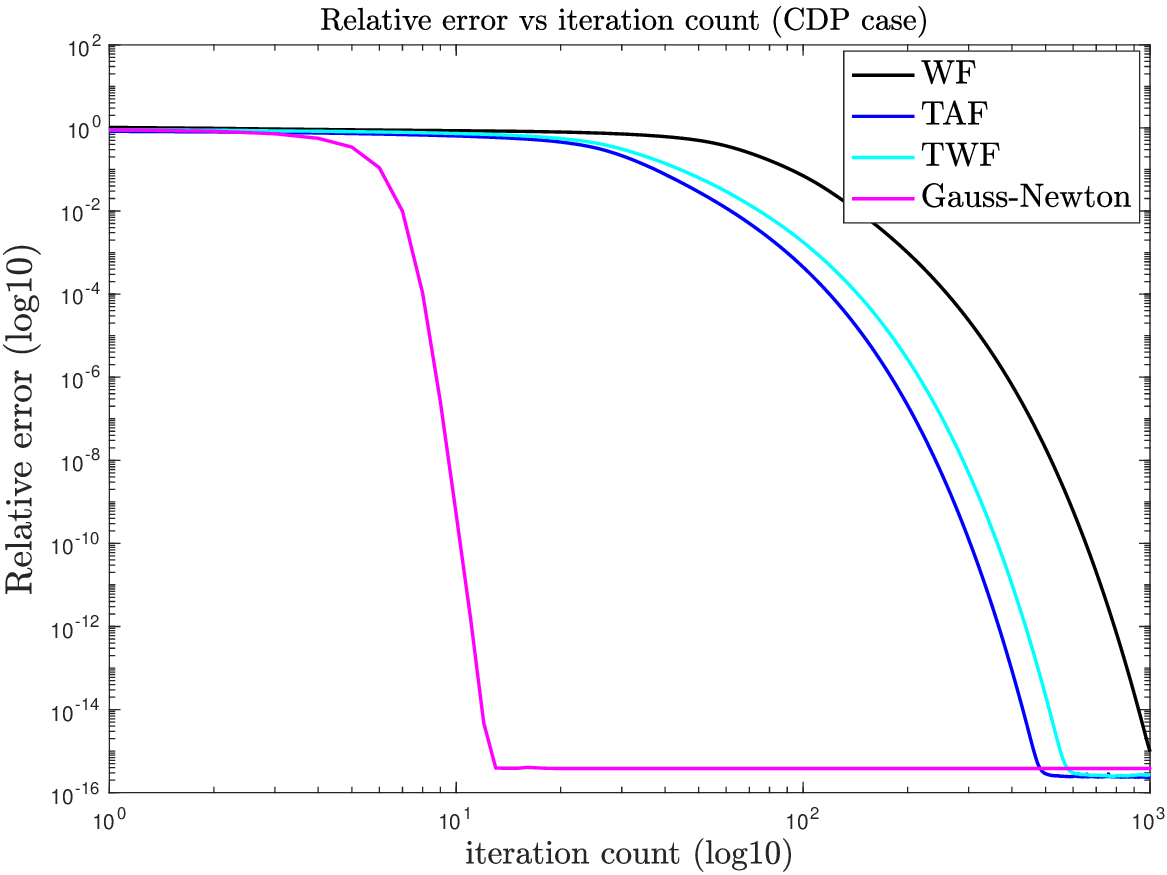}}
\caption{ Relative error versus iteration count for Gauss-Newton, WF, TWF, and TAF methods.: (a) The complex Gaussian case; (b) The CDP case.}
\label{figure:relative_error}
\end{figure}

\begin{figure}[H]
\centering
    \subfigure[]{
     \includegraphics[width=0.45\textwidth]{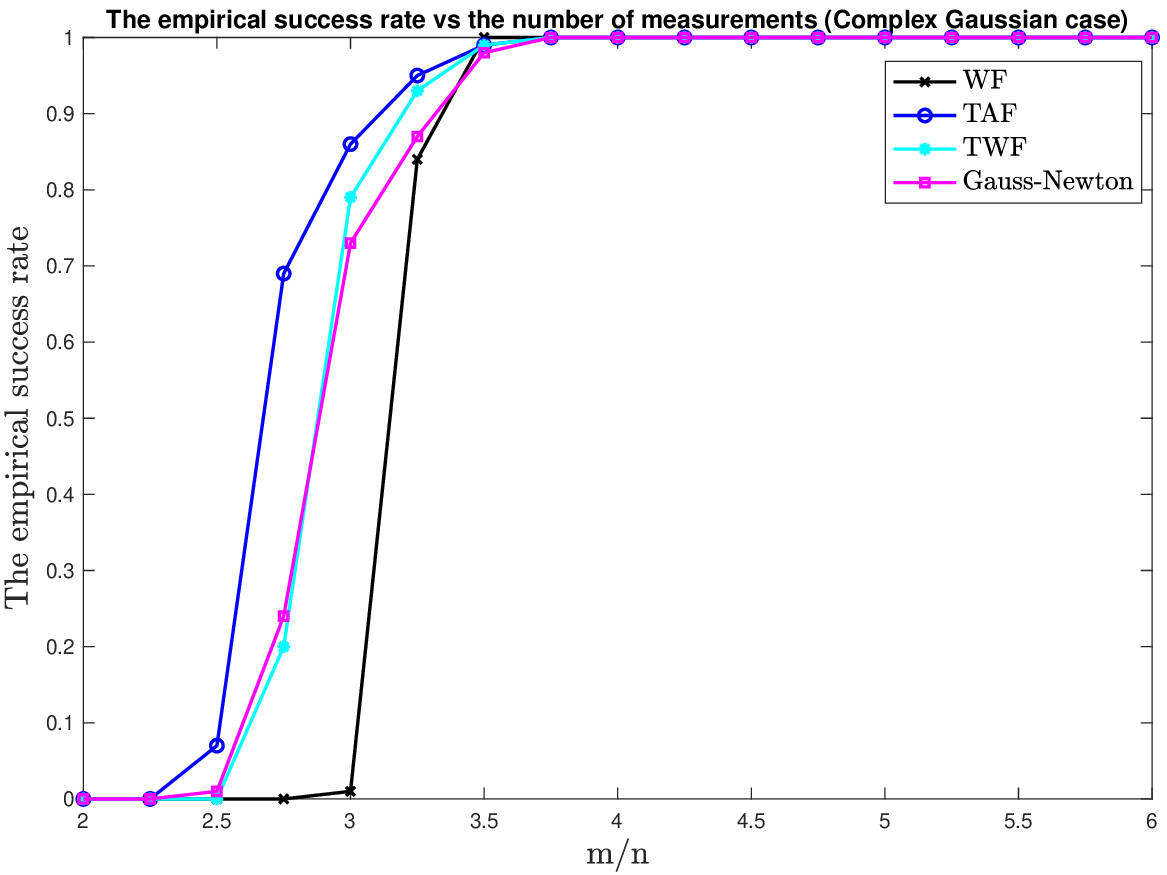}}
\subfigure[]{
     \includegraphics[width=0.45\textwidth]{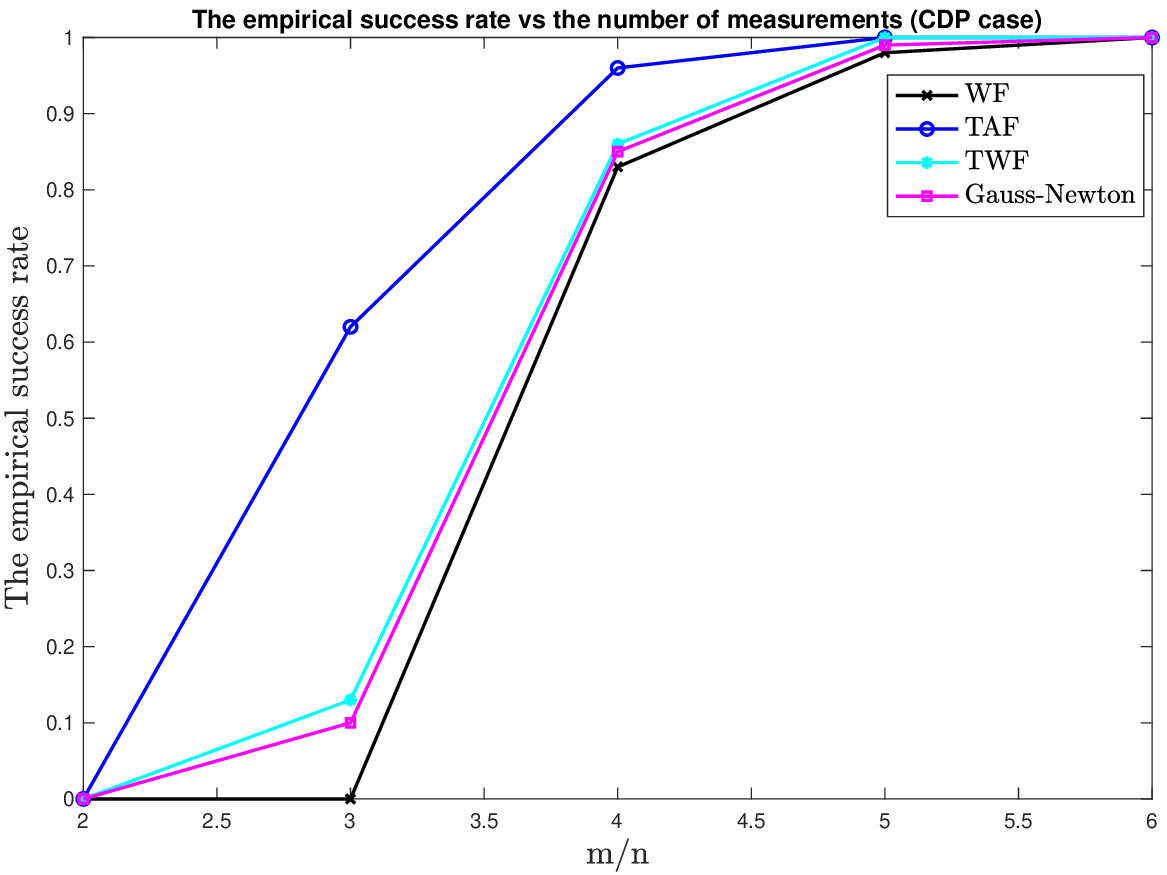}}
\caption{ The empirical success rate for different $m/n$ based on $100$ random trails. (a) Success rate for complex Gaussian case, (b) Success rate for CDP case.}
\label{figure:suc}
\end{figure}

\subsection{Robusteness to Poisson and Gaussian noises}
We now investigate the performance of the Gauss-Newton method under noisy measurements. Two types of noise distributions are considered. The first is Poisson noise, where $y_j=\mbox{Poisson} (\abs{\nj{\va_j,\vxs}}^2)$ for all $j=1,\ldots,m$.  The second is additive white Gaussian noise, where $y_j=\abs{\nj{\va_j,\vxs}}^2+\xi_j$ with $\xi_j$ being i.i.d. Gaussian random variables.  The measurements $\va_j \in \C^n$ are complex Gaussian random vectors.
\begin{example}
In this example, we compare the computational complexity of the Gauss-Newton method with those of WF, TWF, TAF under noisy measurements. We set $n=1000$ and the number of measurements $m=5n$.  For Gaussian noises, $\xi_j \sim 0.1 \cdot \mathcal N(0,1)$ for all $1\le j\le m$.  The relative error versus the number of iterations are presented in  \ref{figure:relative_error1}. We observe that the Gauss-Newton method requires the fewest iterations to converge for both Gaussian and Poisson noise.
\begin{figure}[H]
\centering
\subfigure[]{
     \includegraphics[width=0.45\textwidth]{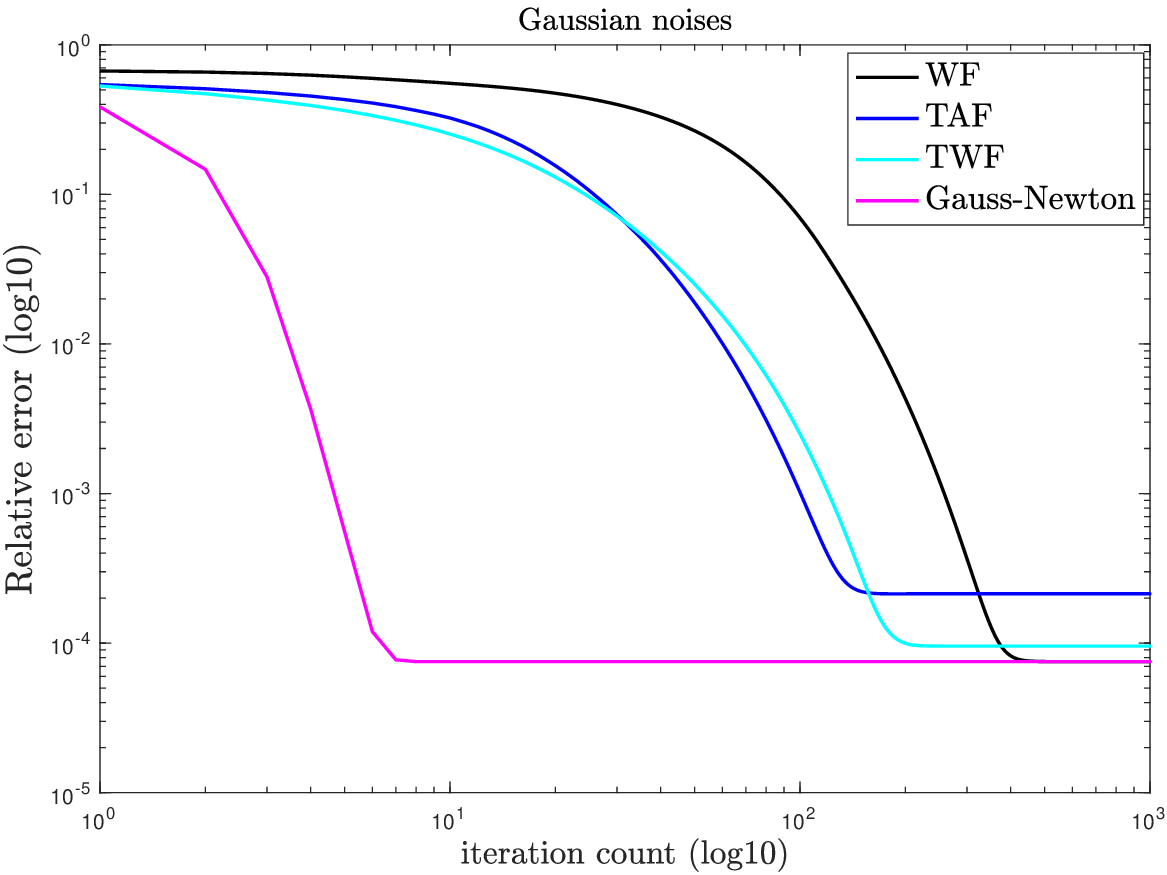}}
\subfigure[]{
     \includegraphics[width=0.45\textwidth]{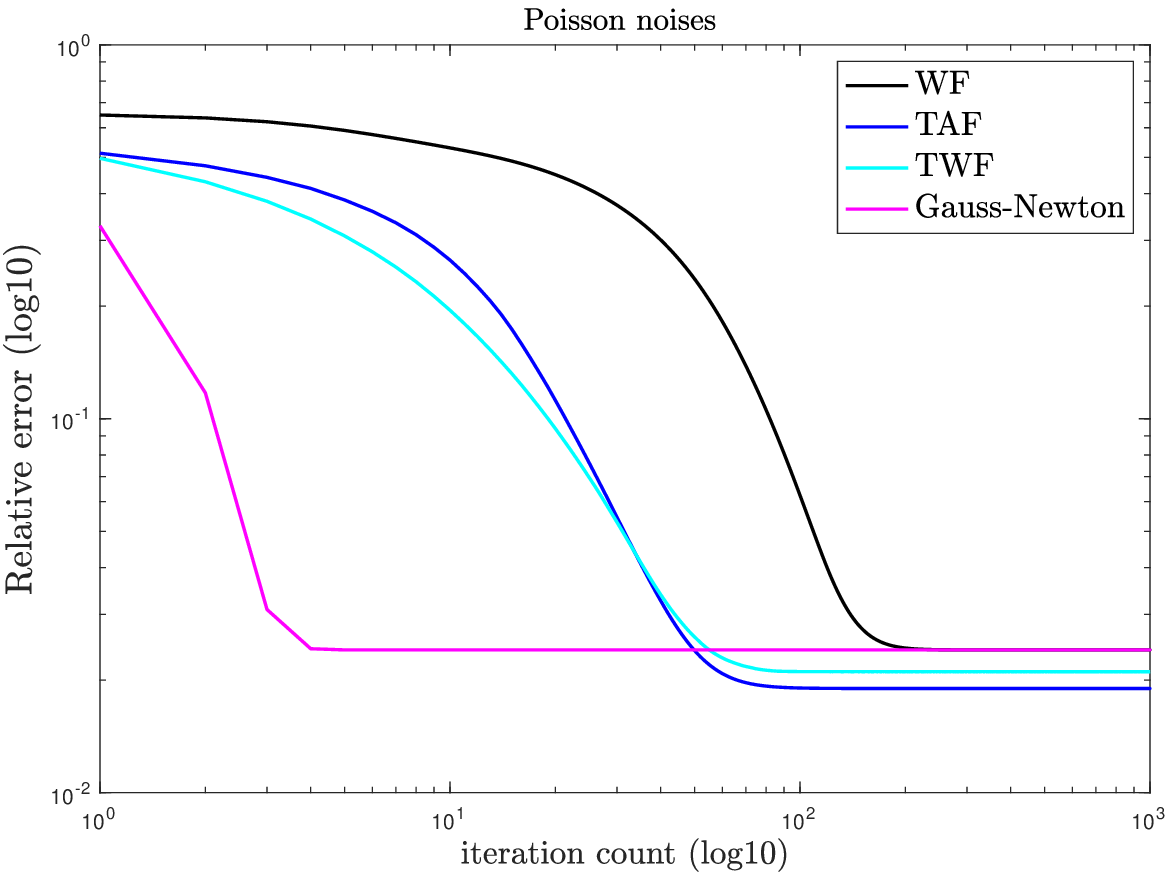}}
\caption{ Relative error versus number of iterations for Gauss-Newton, WF, TWF, and TAF methods under noisy measurements: (a) Gaussian noises; (b) Poisson noises.}
\label{figure:relative_error1}
\end{figure}
\end{example}

\begin{example}
In this example, we introduce varying levels of Gaussian and Poisson noise to examine the relationship between the signal-to-noise ratio (SNR) of the measurements and the mean square error (MSE) of the recovered signal. Specifically, SNR and MSE are evaluated by
\[
\mbox{MSE}:= 10 \log_{10} \frac{{\rm dist }^2(\vz,\vx)}{\norms{\vx}^2} \quad \mbox{and} \quad \mbox{SNR}=10 \log_{10} \frac{\sum_{i=1}^m |\va_j^* \vx|^4}{\norms{\xi}^2},
\]
where $\vz$ is the output of the algorithms. Here, for Poisson noises, we define $\xi_j:=y_j- |\va_j^* \vx|^2$ for all $j=1,\ldots, m$ where $y_j= {\rm Poisson} (|\va_j^* \vx|^2)$.
We set  $n=1000$ and $m=5n$,  and vary the SNR from 20 dB to 60 dB. The results are shown in Figure \ref{figure:SNR}.  We observe that the Gauss-Newton method performs well for noisy phase retrieval in comparison with state-of-the-art algorithms.
\begin{figure}[H]
\centering
\subfigure[]{
     \includegraphics[width=0.45\textwidth]{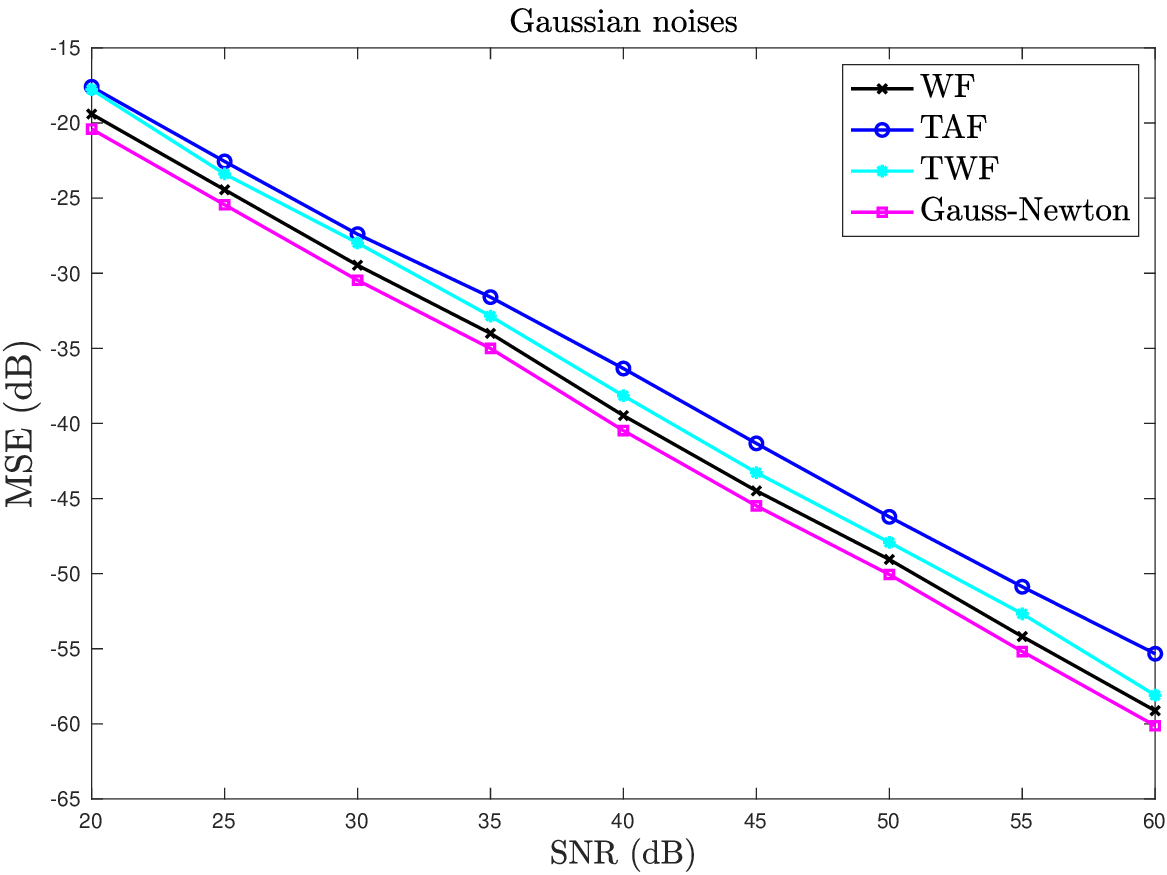}}
\subfigure[]{
     \includegraphics[width=0.45\textwidth]{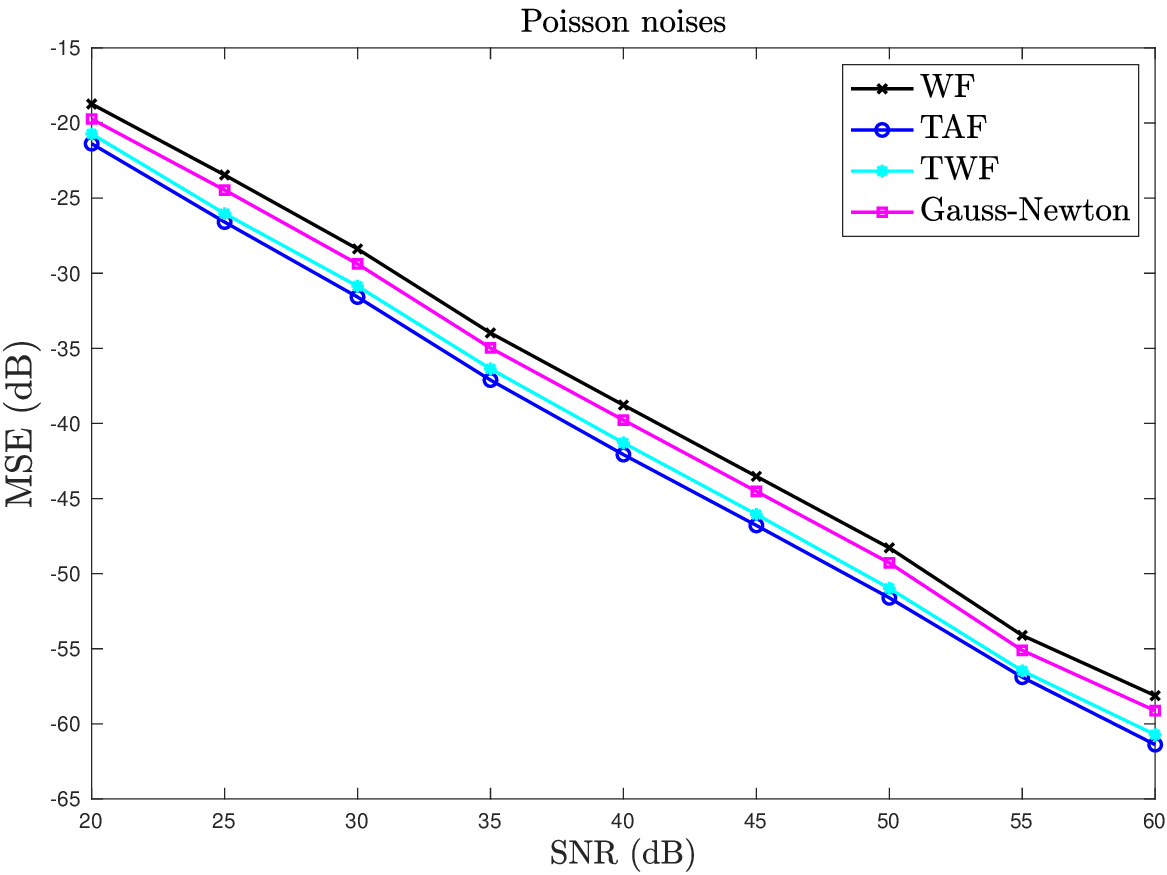}}
\caption{ SNR versus relative MSE on a dB-scale under the noisy measurements: (a) Gaussian noises; (b)  Poisson noises.}
\label{figure:SNR}
\end{figure}
\end{example}

\subsection{Recovery of Natural Image}
Next, we compare the performance of the Gauss-Newton method in recovering a natural image from masked Fourier intensity measurements. The image used is the Milky Way Galaxy with a resolution of  $1080 \times 1920$. The colored image contains RGB channels. We employ $L=18$ random octanary patterns to obtain the Fourier intensity measurements for each R/G/B channel, as described in \cite{WF}. Table \ref{tab:performance_comp_image} lists the averaged time elapsed and the number of iterations required to achieve relative errors of  $10^{-5}$ and $10^{-10}$   across the three RGB channels. The results show that the Gauss-Newton method runs faster than WF, TWF, and TAF, outperforming these algorithms in both the number of iterations and computational time.  Figure  \ref{figure:galaxy} shows the image recovered by the Gauss-Newton method.

\begin{table}[tp]
  \centering
  \fontsize{13}{16}\selectfont
  \caption{Time Elapsed and Number of Iterations among Algorithms on Recovery of Galaxy Image.}
  \label{tab:performance_comp_image}
    \begin{tabular}{|c|c c|cc|}
    \hline
    \multirow{2}{*}{Algorithm}&
    \multicolumn{4}{c|}{The Milky Way Galaxy}\cr\cline{2-5}
    &\multicolumn{2}{c|}{$10^{-5}$ }&\multicolumn{2}{c|}{$10^{-10}$ }\cr \hline
    & \# Passes & Time(s) & \# Passes & Time(s) \cr \hline
   Gauss-Newton & \bf{12} &\bf{103.2} & \bf{17} &\bf{134.1} \cr\hline
    WF &184 & 160.6 & 268 & 224.3 \cr \hline
    TAF &129 & 153.1 &196 & 218.6 \cr \hline
    TWF& 61 & 114.3& 97 &176.2\cr \hline
    \end{tabular}
\end{table}

\begin{figure}[H]
\centering
     \includegraphics[width=0.9\textwidth]{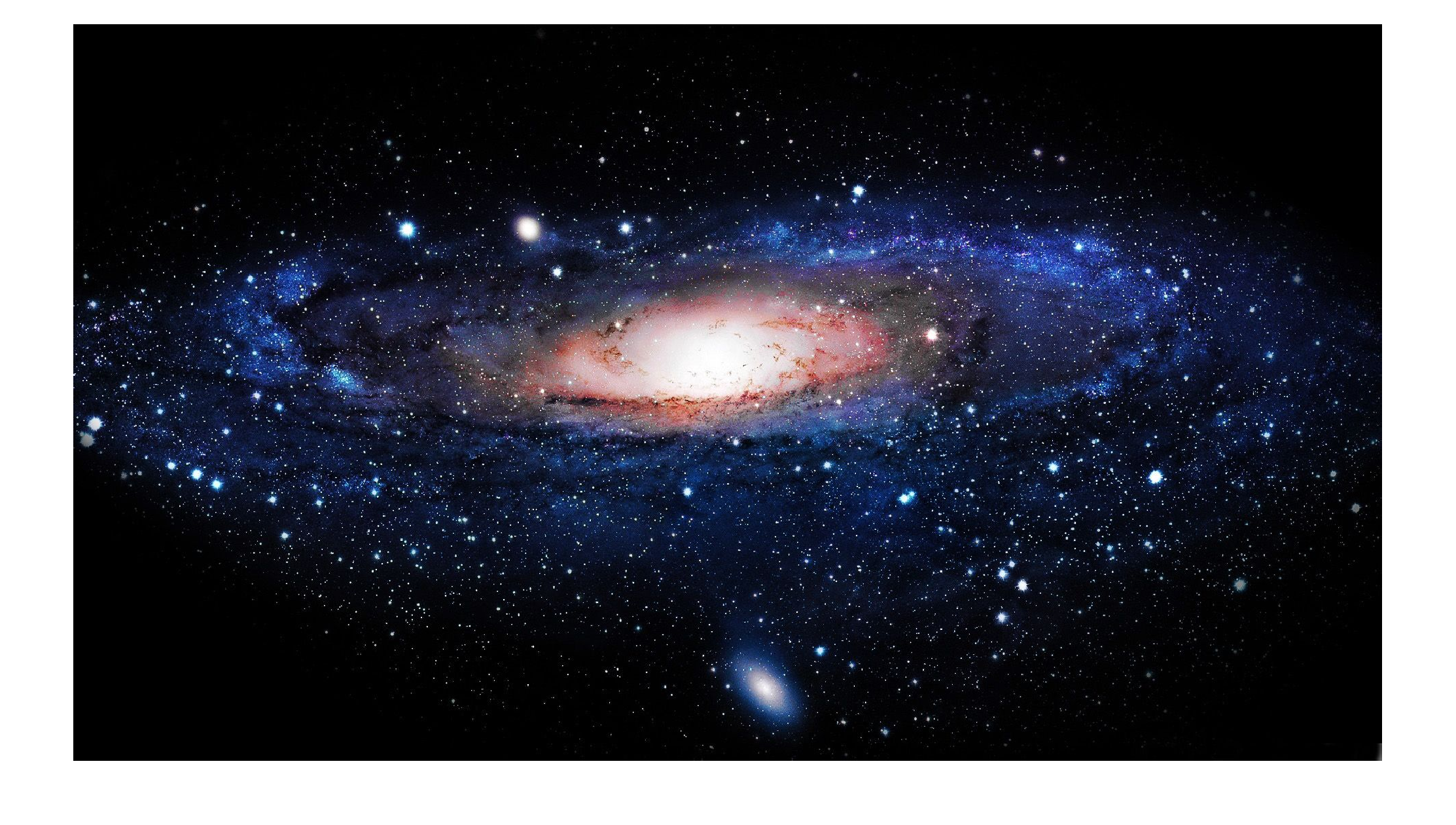}
\caption{ The Milky Way Galaxy image: The Gauss-Newton method with $L=18$ takes $22$ iterations, computation time is $168.8$ s, relative error is  $4.56 \times 10^{-15}$.}
\label{figure:galaxy}
\end{figure}

\section{Convergence analysis}
In this section, we present the proof of the main results.  Without loss of generality, we always assume that $\norm{\vxs}=1$. Before proceeding, we collect a few simple facts. The standard concentration inequality reveals that 
\begin{equation} \label{eq:inajxs}
\max_{1\le j\le m} ~\abs{\va_j^* \vxs  }  \le 5 \sqrt{\log m} \|\vxs\|_2
\end{equation}
with probability exceeding $1-O(m^{-10})$. In addition, one can apply the standard concentration inequality again to show that
\begin{equation} \label{eq:maallaj}
\max_{1\le j\le m} ~ \norm{\va_j} \le \sqrt{6n}
\end{equation}
with probability at least $1-O(me^{-1.5 n})$.

Our theoretical analysis employs leave-one-out arguments, a technique used to demonstrate the weak statistical dependency between the iterates and the design vectors. The general recipe of leave-one-out arguments for non-convex algorithms can be found in \cite{macong}. In below, we first demonstrate that if the current iteration $\vz_k$ stays in the region of incoherence and contraction, then the estimation error of the next iteration $\vz_{k+1}$ shrinks, which implies $\vz_{k+1}$ stays in the region of contraction. 
Next, we employ leave-one-out arguments to show that  $\vz_{k+1}$ also stays in the region of incoherence.  Finally, via an induction argument, the proof is complete  by verifying the desired properties of the initial guess $\vz_0$.

\subsection{Error Contraction}
In this section, we demonstrate that if the current iteration $\vz_k$ remains within the region of incoherence and contraction, then the estimation error of the next iteration $\vz_{k+1}$ will decrease.

\begin{lemma}  \label{le:conun1}
For any sufficiently small constant $\epsilon_0>0$, suppose that $m\ge C_0 \epsilon_0^{-4} n\log^3 m$ for a constant $C_0>0$.  Then with probability at least $1- O(m^{-10})$, 
\[
\dist(\vz_{k+1}, \vxs) \le 2 \dist^2(\vz_{k}, \vxs) +\epsilon_0 \dist(\vz_{k}, \vxs)
\]
holds simultaneously for all $\vz_k$ obeying 
\begin{subequations} \label{eq:assumled}
\begin{gather} 
\norm{\vz_k- \vxs e^{i\phi(\vz_k)} } \le \delta ,  \label{eq:assumleda}\\
\max_{1\le j\le m} ~\abs{\va_j^* (\vz_k- \vxs e^{i\phi(\vz_k)})} \le C_1 \sqrt{\log m}.   \label{eq:assumledb}
\end{gather}
\end{subequations}
Here, $\vz_{k+1}$ is obtained by the Gauss-Newton update rule \eqref{eq:al1zk}, and $0<\delta\le 0.01, C_1>0$ are universal constants.
\end{lemma}
\begin{proof}
See Section \ref{sec:Le2.1}.
\end{proof}

\subsection{Leave-One-Out Sequences}
As shown in Lemma \ref{le:conun1},  if the current iteration $\vz_k$ obeys \eqref{eq:assumled}, then the next iteration $\vz_{k+1}$ will also satisfy \eqref{eq:assumleda}. However, establishing the incoherence condition  \eqref{eq:assumledb} is more complicated. This complexity arises partly from the statistical dependence between $\vz_k$ and the sampling vectors $\va_j$. To address this issue, we employ a leave-one-out approach to introduce an auxiliary sequence of iterations, using all but one sample for consideration.

To be precise, for each $1\le l\le m$, we consider the leave-one-out empirical loss function

\[
f^{(l)}(\vz):=\frac1{m} \sum_{j\neq l} \xkh{\abs{\va_j^* \vz}^2- y_j }^2:=  \frac1{m} \sum_{j\neq l} \xkh{F_j(\vz)}^2.
\]
Then the auxiliary iterations $\vz_{k}^{(l)}$ is constructed by running Gauss-Newton method with respect to $f^{(l)}(\vz)$. More specifically, if the current iteration is $\vzl_k$ then the next iteration is
\[
\vz_{k+1}^{(l)}=\vz_{k}^{(l)} + \delta_k^{(l)}
\]
with 
\begin{equation} \label{eq:minpro11}
\delta_k^{(l)}=\mbox{argmin}_{\delta \in \C^n} \quad \norm{ A^{(l)}(\vz_{k}^{(l)}) \left [ \begin{array}{l} \delta \vspace{0.5em}\\ \overline{ \delta } \end{array}\right]  + F^{(l)}(\vz_{k}^{(l)}) }^2  \qquad \mbox{s.t.} \qquad \Im(\delta^* \vz_k^{(l)})=0
\end{equation}
where
\begin{equation} \label{eq:Azkl}
A^{(l)}(\vz):=\frac1{\sqrt{m}}   \left [ \begin{array}{cc} \vz^*\va_1\va_1^*,&  \vz^\T \bar{\va}_1\va_1^\T \\
\vdots & \vdots\\
\vz^*\va_{l-1}\va_{l-1}^*,&  \vz^\T \bar{\va}_{l-1}\va_{l-1}^\T \\
\vz^*\va_{l+1}\va_{l+1}^*,&  \vz^\T \bar{\va}_{l+1}\va_{l+1}^\T \\
\vdots & \vdots\\
\vz^*\va_m\va_m^*,&  \vz^\T \bar{\va}_m\va_m^\T \end{array} \right ] \in \C^{(m-1)\times 2n}
\end{equation}
and 
\begin{equation} \label{eq:Fzkl}
 F^{(l)}(\vz)=\frac1{\sqrt{m}}  (\abs{\va_1^* \vz}^2- y_1,\ldots, \abs{\va_{l-1}^* \vz}^2- y_{l-1},\abs{\va_{l+1}^* \vz}^2- y_{l+1},\ldots, \abs{\va_m^* \vz}^2- y_m  )^\T.
\end{equation}
In addition, the spectral initialization $\vz_0^{(l)}$ is computed based on the rescaled leading eigenvector of the leave-one-out matrix
\[
Y^{(l)}:=\frac1m \sum_{j\neq l} y_j \va_j\va_j^*.
\]
Clearly, the entire sequence $\dkh{\vzl_k}_{k\ge 0}$ is independent of the $l$-th sampling vector $\va_l$. This auxiliary procedure is formally described in Algorithm \ref{al:2}.
\begin{algorithm}[H]  
\caption{The $l$-th leave-one-out sequence for complex phase retrieval}
\label{al:2}
\begin{algorithmic}[H]
\Require
Measurement vectors: $\va_j \in \C^n, j=1,\ldots,m, j\neq l $; Observations: $y_j \in \C,  j=1,\ldots,m, j\neq l$; the maximum number of iterations $T$.   \\
\textbf{Spectral initialization:} Let $\lambda_1(Y^{(l)})$ and $\tz_0^{(l)} \in \C^n$ be the leading eigenvalue and eigenvector of 
\[
Y^{(l)}:=\frac1m \sum_{j\neq l} y_j \va_j\va_j^*,
\]
respectively. And set $\vzl_0=\sqrt{\lambda_1(Y^{(l)})/2}\cdot \tz_0^{(l)}$. \\
\textbf{Gauss-Newton updates:}  for $k=0,1,\ldots, T-1$ 
\[
\vzl_{k+1}=\vzl_k- \Al(\vzl_k)^{\dag} \Fl(\vzl_k)(1:n),
\]
where $\Al(\vzl_k)$ and $\Fl(\vzl_k)$ are given in \eqref{eq:Azkl} and \eqref{eq:Fzkl}, respectively.
\Ensure
The vector $ \vzl_T $.
\end{algorithmic}
\end{algorithm}

\subsection{Establishing the Incoherence Condition by Induction}
As previously indicated, to prove the main theorem, it is sufficient to demonstrate that the iterates  $\dkh{\vz_k}_{k\ge 0}$ satisfy \eqref{eq:assumledb} with high probability. Our proof will follow an inductive approach. Therefore, we outline all the induction hypotheses as follows:
\begin{subequations} \label{eq:hypoind}
\begin{gather} 
\dist(\vz_k,\vxs) \le \delta ,  \label{eq:hypoinda}\\
\max_{1\le l \le m} ~\dist(\vz_k, \vzl_k) \le C_2 \sqrt{\frac{\log m} m},  \label{eq:hypoindb} \\
\max_{1\le l \le m} ~\abs{\va_l^* (\vz_k- \vxs e^{i\phi(\vz_k)})} \le C_1 \sqrt{\log m}.   \label{eq:hypoindc}
\end{gather}
\end{subequations}
Here, $C_1, C_2>0$ are some universal constants, and $\delta$ is a constant obeying $0<\delta \le 0.01$.

According to Lemma \ref{le:conun1},  if the conditions \eqref{eq:hypoind} hold for the $k$-th iteration, then with probability at least $1-O(m^{-10})$, one has
\begin{equation} \label{eq:diszkx0}
\dist(\vz_{k+1},\vxs) \le \delta.
\end{equation}
This implies that  the hypothesis \eqref{eq:hypoinda}  holds for  the $(k+1)$-th iteration, given that \eqref{eq:hypoind} is satisfied up to the $k$-th iteration.   The subsequent lemma demonstrates that \eqref{eq:hypoindb} is also valid for the $(k+1)$-th iteration.

\begin{lemma} \label{le:diszkl}
Suppose the $m\ge C_0n\log^3m$ for some sufficiently large constant $C_0>0$. Assume that the hypotheses in  \eqref{eq:hypoind} hold for the $k$-th iteration. Then with probability at least $1-O(m^{-10})-O(me^{-1.5n})$, one has
\begin{equation} \label{eq:diszkl0}
\max_{1\le l \le m} ~\dist(\vz_{k+1}, \vzl_{k+1}) \le C_2 \sqrt{\frac{\log m} m}.
\end{equation}
\end{lemma}
\begin{proof}
The proof, which relies heavily on the decomposition theory for pseudo-inverses, is deferred to Section \ref{pro:zkl}.
\end{proof}

A direct consequence of Lemma \ref{le:diszkl} is the incoherence between $\vz_{k+1}-  \vxs e^{i\phi(\vz_{k+1})}$ and $\dkh{\va_l}$, namely,
\[
\max_{1\le l \le m} ~\abs{\va_l^* (\vz_{k+1}-  \vxs e^{i\phi(\vz_{k+1})})} \le C_1 \sqrt{\log m}.
\]
To see this,  define
\begin{equation} \label{eq:dezkl}
\phi_{\vz_k, \vzl_k}^{\mutual}=\mbox{argmin}_{\phi \in [0,2\pi)} \quad \norm{\vz_k-\vzl_k e^{i\phi}}, \qquad  \tzl_k=\vzl_k e^{i \phi_{\vz_k, \vzl_k}^{\mutual}}.
\end{equation}
One can use the triangle inequality to show that with probability at least $1-O(m^{-9})-O(me^{-1.5n})$, it holds

\begin{eqnarray} \label{eq:nearindaz}
&& \\
&& \max_{1\le l \le m} ~\abs{\va_l^* (\vz_{k+1}-  \vxs e^{i\phi(\vz_{k+1})})}  \notag\\
&=&  \max_{1\le l \le m} ~\abs{\va_l^* (\vz_{k+1}  e^{- i\phi(\vz_{k+1})} -  \vxs )}   \notag \\
& \le & \max_{1\le l \le m} ~\Abs{\va_l^* \xkh{ \vz_{k+1}  e^{- i\phi(\vz_{k+1})}  -  \vzl_{k+1}  e^{- i\phi(\vzl_{k+1})} }}  +  \max_{1\le l \le m} ~\Abs{\va_l^* \xkh{  \vzl_{k+1}  e^{- i\phi(\vzl_{k+1})}  -\vxs }}  \notag  \\
& \overset{\text{(i)}}{\le}  & \max_{1\le l \le m} ~ \norm{\va_l} \norm{\vz_{k+1}  e^{- i\phi(\vz_{k+1})}  -  \vzl_{k+1}  e^{- i\phi(\vzl_{k+1})} }+ 5 \sqrt{\log m} \norm{ \vzl_{k+1}  e^{- i\phi(\vzl_{k+1})}  -\vxs }  \notag \\
&\le & \xkh{\sqrt{6n}+  5 \sqrt{\log m} } \norm{\vz_{k+1}  e^{- i\phi(\vz_{k+1})}  -  \vzl_{k+1}  e^{- i\phi(\vzl_{k+1})} } +  5 \sqrt{\log m} \norm{  \vz_{k+1}  e^{- i\phi(\vz_{k+1})}   -\vxs }  \notag \\
& \overset{\text{(ii)}}{\le}  & \xkh{\sqrt{6n}+  5 \sqrt{\log m} } \norm{\vz_{k+1}-  \tzl_{k+1}}+ 5 \sqrt{\log m} \norm{ \vz_{k+1}-  \vxs e^{i\phi(\vz_{k+1})}}  \notag \\
& \overset{\text{(iii)}}{\le}  &  \xkh{\sqrt{6n}+  5 \sqrt{\log m} }  \cdot C_2 \sqrt{ \frac{\log m}m } + 5 \sqrt{\log m}  \cdot \delta   \notag\\
&\le & C_1 \sqrt{\log m} \notag
\end{eqnarray}
for some constant $C_1\ge 8C_2+ 5\delta$. Here, (i) follows from the Cauchy-Schwarz inequality and the independent between $\va_l$ and $\vzl_{k+1}$, (ii) arises from Lemma \ref{le:8.4} by setting $\vz_1=\vz_{k+1} e^{-i\phi(\vz_{k+1})} , \vz_2= \tzl_{k+1} e^{-i\phi(\vz_{k+1})} $ since $\norm{\vz_{k+1} e^{-i\phi(\vz_{k+1})}-\vxs} \le \delta$ and 
\[
\norm{\tzl_{k+1}  e^{-i\phi(\vz_{k+1})}  -\vxs} \le \norm{\vz_{k+1} - \tzl_{k+1}  }+\norm{\vz_{k+1} e^{-i\phi(\vz_{k+1})}-\vxs}  \le 2\delta < \frac14.
\]
And (iii) comes from the bound \eqref{eq:diszkl0} and the condition \eqref{eq:diszkx0}.

Using mathematical induction, we demonstrate that if the current iteration $\vz_k$ satisfies the hypotheses \eqref{eq:hypoind}, then the next iteration $\vz_{k+1}$ will also satisfy these hypotheses with high probability. To complete the proof, it remains to verify that the hypotheses hold for the base case ($k=0$), specifically that the spectral initialization satisfies \eqref{eq:hypoind}. This can be established using Wedin's sin$\Theta$ theorem. 

%The following lemmas are adapted from Lemma 5 in \cite{macong}. For the self-contain, we give the outline of the proof.
\begin{lemma} \label{le:z0x}
For any fixed constant $\delta>0$, suppose $m\ge C_0 n \log m$ for some large constant $C_0>0$. Then with probability exceeding $1-O(m^{-10})$, the vectors $\vz_0 $ given in Algorithm \ref{al:1} obeys
\[
\dist(\vz_0,\vxs) \le \delta \|\vxs\|_2.
\]
\end{lemma}
\begin{proof}
See Section \ref{sec:z0x}.
\end{proof}

The hypothesis \eqref{eq:hypoindb} can also be checked by Wedin's  sin$\Theta$ theorem. 
\begin{lemma} \label{le:z0zl}
Suppose that $m\ge C_0n\log^3 m$ for some large constant $C_0>0$. Then with probability exceeding $1-O(m^{-9})$, one has
\[
\max_{1\le l \le m} ~\dist(\vz_0, \vzl_0) \le C_2 \sqrt{\frac{\log m} m}.
\]
\end{lemma}
\begin{proof}
The proof can be easily adapted from the proof of Lemma 6 in \cite{macong} to the complex case; therefore, we omit it here.
\end{proof}

Based on Lemma \ref{le:z0x} and Lemma \ref{le:z0zl}, the hypothesis \eqref{eq:hypoindc} for $k=0$ can be proved using the same argument as in the derivation of \eqref{eq:nearindaz}, and is therefore omitted.

\subsection{Proof of the Theorem \ref{th:main}}
With Lemmas 5.1-5.4 in place, we are ready to prove the main result Theorem \ref{th:main}.
\begin{proof} [Proof of the Theorem \ref{th:main}]
Without loss of generality, we assume that $\norm{\vxs}=1$. 
Set $T_0:= n$. We divide $\vz_k$ into two stages: $0\le k\le T_0$ and $k>T_0$.  For the first stage,  Lemma \ref{le:z0x} and Lemma \ref{le:z0zl} show that the hypotheses \eqref{eq:assumled} hold for $k=0$ with probability at least $1-O(m^{-9})$.  Combining this with Lemma \ref{le:conun1},  the results of Theorem \ref{th:main} hold for $k=0$ with probability at least $1-O(m^{-9})$.  By invoking Lemma \ref{le:conun1} and Lemma \ref{le:diszkl} recursively for $T_0$ times and taking a union bound, one sees that \eqref{eq:main1} and \eqref{eq:main2} in Theorem \ref{th:main} hold for all $0\le k\le T_0$ with probability exceeding $1-O(m^{-8})-O(me^{-1.5n})$. 

We next turn to the second stage  $k>T_0$. Note that Theorem \ref{th:main} holds for all $0\le k\le T_0$. Therefore,  there exists a constant $0<\rho<1$ such that it holds
\[
\dist(\vz_{T_0+1}, \vxs) \le \rho^{T_0+1} \dist(\vz_{0}, \vxs) \le \frac{1}{n^2}
\]
with probability exceeding $1-O(m^{-8})$. Applying the Cauchy-Schwarz inequality and the fact \eqref{eq:maallaj}, one has
\begin{eqnarray*}
\max_{1\le j\le m} ~\abs{\va_j^* (\vz_{T_0+1}- \vxs e^{i\phi(\vz_{T_0+1})})} & \le &   \max_{1\le j\le m} ~ \norm{\va_j} \cdot \dist(\vz_{T_0+1}, \vxs) \\
&\le  & \sqrt{6n} \cdot \frac{1}{n^2} \\
& \le &  C_1 \sqrt{\log m}
\end{eqnarray*}
with probability at least $1-O(m^{-8})-O(me^{-1.5 n})$.  Using Lemma \ref{le:conun1} again, one can establish the results in Theorem \ref{th:main} for $k=T_0+1$. Repeating the above argument, one can prove Theorem \ref{th:main} for all $k>T_0$ with probability at least $1-O(m^{-8})-O(me^{-1.5 n})$. This completes the proof.

\end{proof}

\section{Discussions}
This paper considers the convergence of the Gauss-Newton method for phase retrieval in the complex setting. Due to the rank-deficiency of the Gauss-Newton matrices, each Gauss-Newton step is restricted to move orthogonally to certain trivial directions, corresponding to the minimal-norm Gauss-Newton method. By employing leave-one-out techniques, we establish an asymptotic quadratic convergence rate for the minimal-norm Gauss-Newton method  without the need of sample splitting.

There are some interesting problems for future research. First, due to the heavy-tailed behavior of the fourth powers of Gaussian random variables, we have only demonstrated the asymptotic quadratic convergence of the Gauss-Newton method. Proving the quadratic convergence rate through more sophisticated methods would be an interesting challenge.  Second,  our results indicate that the Gauss-Newton method succeeds with $m\ge O(n\log^3n)$ samples, while our simulations suggested that $O(n\log n)$ or $O(n)$ may suffice. Improving the sampling complexity to $O(n\log n)$ or even $O(n)$ would be a valuable advancement. Finally, it has been numerically shown that the Gauss-Newton method is also efficient for solving the Fourier phase retrieval problem, particularly when the measurements follow the coded diffraction pattern (CDP) model. Providing theoretical justifications for this efficiency would be of significant practical interest.

\section{Appendix A: Proofs of Technical Results} \label{sec:diszkl}

\subsection{Proof of Lemma \ref{le:Hzlowup}}\label{sec:Hzlowup} \hfill\\

From the definition of the matrix $A(\vz)$ given in \eqref{eq:Azk}, one has 
\begin{equation*} 
A(\vz)^* A(\vz)= \left [ \begin{array}{ll}  \frac 1m \sum_{j=1}^m \abs{\va_j^* \vz}^2 \va_j\va_j^* &   \frac 1m \sum_{j=1}^m (\va_j^*\vz)^2 \va_j \va_j^\T   \\  \frac 1m \sum_{j=1}^m \xkh{\vz^* \va_j}^2 \overline{\va_j} \va_j^* & \frac 1m \sum_{j=1}^m \abs{\va_j^* \vz}^2 \overline{\va_j} \va_j^\T   \end{array}\right].
\end{equation*}
Applying Lemma \ref{le:contra}, with probability exceeding $1-O(m^{-10})$, it holds
\begin{eqnarray*}
 \norm{A(\vz)^* A(\vz)} & \le & \norm{ \frac 1m \sum_{j=1}^m \abs{\va_j^* \vz}^2 \va_j\va_j^* } +\norm{ \frac 1m \sum_{j=1}^m \xkh{\va_j^*\vz}^2 \va_j \va_j^\T} \\
&\le & \norm{\vz\vz^* + \norm{\vz}^2 I_{n}} +2\norm{\vz \vz^\T}  + 2c_0 \sqrt{\frac{n\log^3 m}{m} } \norm{\vz}^2 \\
& \le & 5 \norm{\vz}^2
\end{eqnarray*}
for all $\vz \in \C^n$ obeying $\max_{1\le j\le m} ~\abs{\va_j^* \vz} \le C_1 \sqrt{\log m} \norm{\vz}$, provided $m\ge C_2 n\log^3 m$.  Here, $C_2, c_0>0$ are universal constants with $C_2\ge 4c_0^2$.

For the lower bound, by the definition of $H(\vz)$ given in \eqref{eq:Hz}, it suffices to show 
\[
\frac12 \left [ \begin{array}{l} \vw \vspace{0.4em} \\ \overline{ \vw } \end{array}\right]^* A(\vz)^* A(\vz) \left [ \begin{array}{l} \vw \vspace{0.4em} \\ \overline{ \vw } \end{array}\right] \ge 0.95\norm{\vz}^2  \quad \mbox{for all} \quad \norm{\vw}=1 \quad \mbox{with}  \quad \Im(\vw^* \vz)=0
\]
and for all $\vz$ obeying  
\[
\max_{1\le j\le m} ~\abs{\va_j^* \vz} \le C_1 \sqrt{\log m} \norm{\vz}.
\]
Applying Lemma \ref{le:contra} once again, with probability exceeding $1-O(m^{-10})$, it holds
\begin{eqnarray*}
\frac12\left [ \begin{array}{l} \vw \vspace{0.4em} \\ \overline{ \vw } \end{array}\right]^* A(\vz)^* A(\vz) \left [ \begin{array}{l} \vw \vspace{0.4em} \\ \overline{ \vw } \end{array}\right] &=& \frac 1m \sum_{j=1}^m \abs{\va_j^* \vz}^2 \abs{\va_j^* \vw}^2 + \frac 1m \sum_{j=1}^m \Re(\xkh{\va_j^* \vz}^2 \xkh{\vw^* \va_j}^2 )\\
&\ge &  \norm{\vz}^2 \norm{\vw}^2 + \abs{\vw^* \vz}^2  +  2\Re\xkh{(\vw^* \vz)^2} -   2c_0 \sqrt{\frac{n\log^3 m}{m} } \norm{\vz}^2    \\
&=& \norm{\vz}^2 + 3\abs{\vw^* \vz}^2 -  2c_0 \sqrt{\frac{n\log^3 m}{m} } \norm{\vz}^2\\
&\ge & 0.95 \norm{\vz}^2,
\end{eqnarray*}
where the third line follows from the fact that $\Im(\vw^* \vz)=0$, and the last inequality is valid by taking $m\ge C_0 n\log^3 n$ for some universal constant $C_0>0$. This completes the proof.

$\hfill\square$

\subsection{Proof of Lemma \ref{le:conun1}}\label{sec:Le2.1} \hfill\\

Recall that the Gauss-Newton update rule is 
\[
\vz_{k+1}=\vz_k+\delta_k,
\]
where 
\begin{equation} \label{eq:minpro2}
\delta_k=\mbox{argmin}_{\delta \in \C^n} \quad \norm{ A(\vz_k) \left [ \begin{array}{l} \delta \\ \overline{ \delta } \end{array}\right]  +  F(\vz_k) }^2, \qquad \mbox{s.t.} \qquad \Im(\delta^* \vz_k)=0.
\end{equation}
We treat $\C^n$ as $\R^{2n}$.  Then $S(\vz_k):=\dkh{\vw\in \C^n:  \Im(\vw^* \vz_k)=0}$ forms a subspace of dimension $2n-1$ over $\R^{2n}$.
Consider any matrix $U(\vz_k) \in \C^{n\times (2n-1)}$ whose columns forms an orthonormal basis for the subspace, i.e., $\Re(U_k^* U_l)=\delta_{k,l}$ for any columns $U_k$ and $U_l$. Then the problem \eqref{eq:minpro2} can be reformulated as 
\begin{equation} \label{eq:leastsqxik}
\xi_k=\mbox{argmin}_{\xi \in \R^{2n-1}} \quad \norm{A(\vz_k) \left [ \begin{array}{l} U(\vz_k) \vspace{0.4em} \\ \overline{ U(\vz_k) } \end{array}\right]  \xi + F(\vz_k)  }^2
\end{equation}
and $\delta_k=U(\vz_k) \xi_k$.  Observing that for any vector $\vz_k$ obeys $\norm{\vz_k- \vxs e^{i\phi(\vz_k)} } \le \delta $ and $\max_{1\le j\le m} ~\abs{\va_j^* (\vz_k- \vxs e^{i\phi(\vz_k)})} \le C_1 \sqrt{\log m}$, with probability exceeding $1-O(m^{-10})$,  one has
\begin{eqnarray}
\max_{1\le j\le m} ~\abs{\va_j^* \vz_k} & \le &  \max_{1\le j\le m} ~\abs{\va_j^* \vxs}+ \max_{1\le j\le m} ~\abs{\va_j^* (\vz_k- \vxs e^{i\phi(\vz_k)})}  \notag \\
& \le &  5\sqrt{\log m}+C_1 \sqrt{\log m} \notag \\
& \lesssim &  \sqrt{\log m} \norm{\vz_k}, \label{eq:incaz}
\end{eqnarray}
where we use the fact \eqref{eq:inajxs} in the second inequality, and the the fact that $\norm{\vz_k} \ge 1-\delta \ge 0.99$ in the last inequality.  Armed with the bound \eqref{eq:incaz}, we can apply Lemma \ref{le:Hzlowup} to obtain that with probability at least $1-O(m^{-10})$, the solution to \eqref{eq:leastsqxik} is
\[
\xi_k=- H^{-1}(\vz_k) \left [ \begin{array}{l} U(\vz_k) \vspace{0.4em} \\ \overline{ U(\vz_k) } \end{array}\right]^* A^*(\vz_k)  F(\vz_k),
\]
where the matrix $H(\vz_k)$ is defined in \eqref{eq:Hz}. Therefore, 
\begin{eqnarray*}
\norm{ \left [ \begin{array}{l} \vz_{k+1}-\vxs e^{i\phi(\vz_{k+1}))} \vspace{0.4em} \\ \overline{ \vz_{k+1}-\vxs e^{i\phi(\vz_{k+1})} } \end{array}\right]} &\le &\norm{ \left [ \begin{array}{l} \vz_{k+1}-\vxs e^{i\phi(\vz_k)} \vspace{0.4em} \\ \overline{ \vz_{k+1}-\vxs e^{i\phi(\vz_k)} } \end{array}\right]}  \\
&=& \norm{\left [ \begin{array}{l} \vz_{k}-\vxs e^{i\phi(\vz_k)} \vspace{0.4em} \\ \overline{ \vz_{k}-\vxs e^{i\phi(\vz_k)} } \end{array}\right]  - \left [ \begin{array}{l} U(\vz_k) \vspace{0.4em} \\ \overline{ U(\vz_k) } \end{array}\right]  H^{-1}(\vz_k) \left [ \begin{array}{l} U(\vz_k) \vspace{0.4em} \\ \overline{ U(\vz_k) } \end{array}\right]^* A^*(\vz_k)  F(\vz_k)}.
\end{eqnarray*}
Recall that $U(\vz_k)$ is an orthonormal basis constructed for the space $S(\vz_k):=\dkh{\vw\in \C^n:  \Im(\vw^* \vz_k)=0}$. Using the fact that 
$\Im((\vz_k- \vxs e^{i\phi(\vz_{k})})^* \vz_k)=0$, one has
\[
\left [ \begin{array}{l} \vz_{k}-\vxs e^{i\phi(\vz_k)} \vspace{0.4em} \\ \overline{ \vz_{k}-\vxs e^{i\phi(\vz_k)} } \end{array}\right]=\frac12  \left [ \begin{array}{l} U(\vz_k) \vspace{0.4em} \\ \overline{ U(\vz_k) } \end{array}\right] \left [ \begin{array}{l} U(\vz_k) \vspace{0.4em} \\ \overline{ U(\vz_k) } \end{array}\right]^* \left [ \begin{array}{l} \vz_{k}-\vxs e^{i\phi(\vz_k)} \vspace{0.4em} \\ \overline{ \vz_{k}-\vxs e^{i\phi(\vz_k)} } \end{array}\right].
\]
Similarly, one can also verify that
\begin{equation} \label{eq:idenAk}
A^*(\vz_k) = \frac12  \left [ \begin{array}{l} U(\vz_k) \vspace{0.4em} \\ \overline{ U(\vz_k) } \end{array}\right] \left [ \begin{array}{l} U(\vz_k) \vspace{0.4em} \\ \overline{ U(\vz_k) } \end{array}\right]^*  A^*(\vz_k).
\end{equation}
Therefore, 
\begin{eqnarray}
\norm{ \left [ \begin{array}{l} \vz_{k+1}-\vxs e^{i\phi(\vz_{k+1})} \vspace{0.4em} \\ \overline{ \vz_{k+1}-\vxs e^{i\phi(\vz_{k+1})} } \end{array}\right]} & \le  & \left\| \left [ \begin{array}{l} U(\vz_k) \vspace{0.4em} \\ \overline{ U(\vz_k) } \end{array}\right]  H^{-1}(\vz_k) \left [ \begin{array}{l} U(\vz_k) \vspace{0.4em} \\ \overline{ U(\vz_k) } \end{array}\right]^* A^*(\vz_k)   \right.  \notag\\
&& \left. \cdot \xkh{ \frac12 A(\vz_k) \left [ \begin{array}{l} U(\vz_k) \vspace{0.4em} \\ \overline{ U(\vz_k) } \end{array}\right] \left [ \begin{array}{l} U(\vz_k) \vspace{0.4em} \\ \overline{ U(\vz_k) } \end{array}\right]^* \left [ \begin{array}{l} \vz_{k}-\vxs e^{i\phi(\vz_k)} \vspace{0.4em} \\ \overline{ \vz_{k}-\vxs e^{i\phi(\vz_k)} } \end{array}\right] -  F(\vz_k)  } \right\|_2 \notag\\
&\le & \norm{ H^{-1}(\vz_k)  } \norm{A^*(\vz_k)  } \norm{ A(\vz_k)  \left [ \begin{array}{l} \vz_{k}-\vxs e^{i\phi(\vz_k)} \vspace{0.4em} \\ \overline{ \vz_{k}-\vxs e^{i\phi(\vz_k)} } \end{array}\right] -  F(\vz_k)  }. \label{eq:zkcon1}
\end{eqnarray}
where we use the identity 
\[
H^{-1}(\vz_k) \left [ \begin{array}{l} U(\vz_k) \vspace{0.4em} \\ \overline{ U(\vz_k) } \end{array}\right]^* A^*(\vz_k) A(\vz_k) \left [ \begin{array}{l} U(\vz_k) \vspace{0.4em} \\ \overline{ U(\vz_k) } \end{array}\right]= H^{-1}(\vz_k) H(\vz_k) = I_{2n-1}
\]
in the first inequality and \eqref{eq:idenAk} in the second inequality. The fundamental theorem of calculus together with the fact $F(\vxs e^{i\phi(\vz_k)})=0$ gives
\begin{equation} \label{eq:zkcon2}
F(\vz_k) =F(\vz_k) -F(\vxs e^{i\phi(\vz_k)})= \int_0^1 A(\vz_\tau) ~d\tau  \left [ \begin{array}{l} \vz_{k}-\vxs e^{i\phi(\vz_k)} \vspace{0.4em} \\ \overline{ \vz_{k}-\vxs e^{i\phi(\vz_k)} } \end{array}\right],
\end{equation}
where we denote $\vz_\tau=\vxs e^{i\phi(\vz_k)}  + \tau(\vz_k-\vxs e^{i\phi(\vz_k)} )$. Putting identity \eqref{eq:zkcon2} into \eqref{eq:zkcon1}, we obtain
\begin{equation} \label{eq:up3}
\norm{ \left [ \begin{array}{l} \vz_{k+1}-\vxs e^{i\phi(\vz_{k+1})} \vspace{0.4em} \\ \overline{ \vz_{k+1}-\vxs e^{i\phi(\vz_{k+1})} } \end{array}\right]} \le \norm{ H^{-1}(\vz_k)  } \norm{A^*(\vz_k)  }  \norm{\int_0^1 (A(\vz_k)-A(\vz_\tau)) ~d\tau   \left [ \begin{array}{l} \vz_{k}-\vxs e^{i\phi(\vz_k)} \vspace{0.4em} \\ \overline{ \vz_{k}-\vxs e^{i\phi(\vz_k)} } \end{array}\right] }.
\end{equation}
Lemma \ref{le:Hzlowup} together with \eqref{eq:incaz} yields  
\begin{equation} \label{eq:Hinvaz}
\norm{ H^{-1}(\vz_k)  }  \le  \frac{1}{1.9 \norm{\vz_k}^2}  \qquad \mbox{and} \qquad        \norm{A^*(\vz_k)  } \le \sqrt5 \norm{\vz_k} .
\end{equation}
 We claim that for any $\epsilon_0>0$, with probability at least $1-O(m^{-10})$, it holds
\begin{equation} \label{eq:LipAz}
\norm{A(\vz_k)-A(\vz_\tau) } \le 2(1-\tau)\norm{\vz_{k}-\vxs e^{i\phi(\vz_k)}}+ \frac{\epsilon_0}2,
\end{equation}
provided $m\ge C_0 \epsilon_0^{-4} n\log^3m$.  Putting \eqref{eq:Hinvaz} and \eqref{eq:LipAz} into \eqref{eq:up3}, we obtain
\[
\norm{\vz_{k+1}-\vxs e^{i\phi(\vz_{k+1})}} \le 2 \norm{\vz_{k}-\vxs e^{i\phi(\vz_k)}}^2 +\epsilon_0 \norm{\vz_{k}-\vxs e^{i\phi(\vz_k)}}.
\]
Here, we use the fact that $\norm{\vz_k} \le (1+\delta)\norm{\vxs}\le 1.01$.

It remains to prove the claim \eqref{eq:LipAz}. From the definition of $A(\vz)$, we know
\[
A(\vz_k)-A(\vz_\tau) = \frac1{\sqrt{m}}   \left [ \begin{array}{cc} (\vz_k-\vz_\tau)^*\va_1\va_1^*,&  (\vz_k-\vz_\tau)^\T \bar{\va}_1\va_1^\T \\
\vdots & \vdots\\
(\vz_k-\vz_\tau)^*\va_m\va_m^*,&  (\vz_k-\vz_\tau)^\T \bar{\va}_m\va_m^\T \end{array} \right ].
\]
Therefore, 
\[
\norm{A(\vz_k)-A(\vz_\tau) }^2 =\norm{ \left [ \begin{array}{ll}  \frac 1m \sum_{j=1}^m \abs{\va_j^* (\vz_k-\vz_\tau)}^2 \va_j\va_j^* \vspace{0.4em}&   \frac 1m \sum_{j=1}^m \xkh{\va_j^* (\vz_k-\vz_\tau)}^2 \va_j \va_j^\T   \\  \frac 1m \sum_{j=1}^m \xkh{(\vz_k-\vz_\tau)^* \va_j}^2 \overline{\va_j} \va_j^* & \frac 1m \sum_{j=1}^m \abs{\va_j^* (\vz_k-\vz_\tau)}^2 \overline{\va_j} \va_j^\T   \end{array}\right] }.
\]
Observe that 
\[
\norm{\vz_k-\vz_\tau} =(1-\tau) \norm{\vz_k- \vxs e^{i\phi(\vz_k)} } \le \delta<1
\]
and 
\[
\max_{1\le j\le m} ~  \abs{\va_j^* (\vz_k- \vz_\tau)}  = (1-\tau)  \cdot \max_{1\le j\le m} ~\abs{\va_j^* (\vz_k- \vxs e^{i\phi(\vz_k)})} \le C_1 \sqrt{\log m}
\]
for all $0\le \tau\le 1$. It then follows from Lemma \ref{le:contra2} that with probability exceeding $1-O(m^{-10})$, it holds
\begin{eqnarray*}
\norm{A(\vz_k)-A(\vz_\tau) }^2  & =  & \norm{ \frac 1m \sum_{j=1}^m \abs{\va_j^* (\vz_k-\vz_\tau)}^2 \va_j\va_j^*} + \norm{\frac 1m \sum_{j=1}^m \xkh{\va_j^* (\vz_k-\vz_\tau)}^2 \va_j \va_j^\T }\\
& \le &4\norm{ \vz_k-\vz_\tau }^2 +2 c_0 \sqrt{\frac{n\log^3m}{m}} \\
&\le &   4 (1-\tau)^2  \norm{\vz_{k}-\vxs e^{i\phi(\vz_k)}}^2 +  \frac{\epsilon_0^2}{4}.
\end{eqnarray*}
Here, the last inequality arises from the that  $m\ge C_0 \epsilon_0^{-4} n\log^3 m$ for some sufficiently large constant $C_0 >0$. This implies 
\[
\norm{A(\vz_k)-A(\vz_\tau) } \le 2(1-\tau)\norm{\vz_{k}-\vxs e^{i\phi(\vz_k)}} + \frac{\epsilon_0}{2},
\]
which concludes the proof.

$\hfill\square$

\subsection{Proof of Lemma \ref{le:diszkl}} \label{pro:zkl} \hfill\\

To begin with, we collect a few immediate consequences of the induction hypothese \eqref{eq:hypoind}, which are useful in the subsequent analysis.
\begin{lemma}
Assume that $m\ge C_0 n\log^3 m$ for some constant $C_0>0$.
Under the hypotheses \eqref{eq:hypoind}, with probability at least $1-O(m^{-10})$, one has
\begin{subequations} \label{eq:hypaft}
\begin{gather} 
0.99 \le \norm{\vz_k} \le 1.01 \label{eq:hypafta}\\
0.89 \le \norm{\vzl_k} \le 1.11 \label{eq:hypaftc}\\
\max_{1\le j\le m} ~\abs{\va_j^* \vz_k} \lesssim \sqrt{\log m} \norm{\vz_k},  \label{eq:hypaftb}\\
\max_{1\le j \le m} ~\abs{\va_j^* \vzl_k} \lesssim   \sqrt{\log m}  \norm{\vzl_k}.  \label{eq:hypaftd}
\end{gather}
\end{subequations}
\end{lemma}
\begin{proof}
Regrading the first set of consequences \eqref{eq:hypaft}, by the triangle inequality, one has
\[
0.99 \le 1-\delta \le  \|\vxs\|_2- \norm{\vz_k-\vxs e^{i \phi(\vz_{k})}} \le \norm{\vz_k} \le \|\vxs\|_2+ \norm{\vz_k-\vxs e^{i \phi(\vz_{k})}}  \le 1+\delta \le 1.01.
\]
Here, we use the hypothesis  \eqref{eq:hypoinda} in the above inequality.

For the second set of consequences \eqref{eq:hypaft}, we have
\[
 \norm{\vzl_k}=\norm{\tzl_k} \le \norm{\vz_k-\tzl_k} +\norm{\vz_k} \le C_2 \sqrt{\frac{\log m} m}+ 1.01 \le 1.11, 
\]
provided $m \ge 100 C_2^2 \log m$.  Here, $\tzl_k$ is defined in \eqref{eq:dezkl}, and the second inequality comes from the hypothesis \eqref{eq:hypoindb}. The lower bound can be established similarly.

For the third set of consequences \eqref{eq:hypaft}, we invoke the triangle inequality once again to deduce that 
\[
\max_{1\le j\le m} ~\abs{\va_j^* \vz_k} \le \max_{1\le j\le m} ~\abs{\va_j^* (\vz_k- \vxs e^{i\phi(\vz_k)} ) } + \max_{1\le j\le m} ~\abs{\va_j^* \vxs  } \overset{\text{(i)}}{\le}  (C_1+5) \sqrt{\log m} \overset{\text{(ii)}}{ \lesssim} \sqrt{\log m} \norm{\vz_k}.
\]
Here, (i) utilizes the induction hypothesis  \eqref{eq:hypoindc} and the standard Gaussian concentration, namely,  $\max_{1\le j\le m} ~\abs{\va_j^* \vxs  }  \le 5 \sqrt{\log m}$ with probability exceeding $1-O(m^{-10})$, and (ii) comes from the fact \eqref{eq:hypafta}.

Finally, for the last set of consequences \eqref{eq:hypaft}, one has
\begin{eqnarray*}
\max_{1\le j \le m} ~\abs{\va_j^* \vzl_k}= \max_{1\le j \le m} ~\abs{\va_j^* \tzl_k}& \le &  \max_{1\le j \le m} ~\abs{\va_j^* \vz_k}+  \max_{1\le j \le m} ~\abs{\va_j^* (\vz_k-\tzl_k)} \\
&\le & \sqrt{\log m} + \sqrt{6m} \cdot \sqrt{\frac{\log m}{m}}  \\
&\lesssim &  \sqrt{\log m},
\end{eqnarray*}
where the second inequality follows from the fact \eqref{eq:hypaftb} and the hypothesis \eqref{eq:hypoindb}.

\end{proof}

\begin{proof}[Proof of Lemma \ref{le:diszkl}]
Set 
\begin{equation} \label{eq:lambda}
\lambda=\sqrt m \cdot \max_{1\le l \le m} ~\dist(\vz_{k+1},\vzl_{k+1}).
\end{equation}
It then suffices to show $\lambda \le C_2\sqrt{\log m}$ with probability at least $1-O(m^{-10})$. To this end,  we define
\[
\phi_{\vz_k, \vzl_k}^{\mutual}: =\mbox{argmin}_{\phi \in [0,2\pi)} \quad \norm{\vz_k-\vzl_k e^{i\phi}},
\]
and let $\tzl_k=\vzl_k e^{i \phi_{\vz_k, \vzl_k}^{\mutual}}$.   With these notations,  we have
\begin{equation} \label{eq:diszklp1}
 \dist(\vz_{k+1}, \vzl_{k+1}) \le \norm{\vz_{k+1}- \vzl_{k+1} e^{i \phi_{\vz_k, \vzl_k}^{\mutual}}}.
\end{equation}
For the iteration $\vz_{k+1}$, taking any matrix $U(\vz_k) \in \C^{n\times (2n-1)}$ whose columns forms an orthonormal basis for the subspace $S(\vz_k):=\dkh{\vw\in \C^n:  \Im(\vw^* \vz_k)=0}$, we  have
\begin{equation} \label{eq:upzkp1}
\vz_{k+1}=\vz_k-U(\vz_k)H^{-1}(\vz_k) \left [ \begin{array}{l} U(\vz_k) \vspace{0.4em} \\ \overline{ U(\vz_k) } \end{array}\right]^* A^*(\vz_k)  F(\vz_k)
\end{equation}
where 
\begin{equation}
H(\vz_k)=\left [ \begin{array}{l} U(\vz_k) \vspace{0.4em} \\ \overline{ U(\vz_k) } \end{array}\right]^* A(\vz_k)^* A(\vz_k) \left [ \begin{array}{l} U(\vz_k) \vspace{0.4em} \\ \overline{ U(\vz_k) } \end{array}\right] \in \R^{(2n-1)\times (2n-1)}.
\end{equation}
Here, $A(\vz_k)$ and $F(\vz_k)$ are given in \eqref{eq:Azk} and \eqref{eq:Fzk}, respectively. For the iteration $\vzl_{k+1}$, the Gauss-Newton update rule for the leave-one-out version is $\vzl_{k+1}=\vzl_{k}+\delta_k^{(l)}$, where
\begin{equation*} 
\delta_k^{(l)}=\mbox{argmin}_{\delta \in \C^n} \quad \norm{ A^{(l)}(\vz_{k}^{(l)}) \left [ \begin{array}{l} \delta \vspace{0.5em}\\ \overline{ \delta } \end{array}\right]  + F^{(l)}(\vz_{k}^{(l)}) }^2  \qquad \mbox{s.t.} \qquad \Im(\delta^* \vz_k^{(l)})=0.
\end{equation*}
One can verify that 
\[
\delta_k^{(l)}e^{i \phi_{\vz_k, \vzl_k}^{\mutual}} = \mbox{argmin}_{\delta \in \C^n} \quad \norm{ A^{(l)}(\tzl_k) \left [ \begin{array}{l} \delta \vspace{0.5em}\\ \overline{ \delta } \end{array}\right]  + F^{(l)}(\tzl_k) }^2  \qquad \mbox{s.t.} \qquad \Im(\delta^* \tzl_k)=0.
\]
Therefore, one has
\begin{equation}    \label{eq:upzklp1}
\vzl_{k+1} e^{i \phi_{\vz_k, \vzl_k}^{\mutual}} =\tzl_k-U(\tzl_k) (\Hl(\tzl_k))^{-1} \left [ \begin{array}{l} U(\tzl_k) \vspace{0.4em} \\ \overline{ U(\tzl_k) } \end{array}\right]^* (\Al(\tzl_k))^*  \Fl(\tzl_k).
\end{equation}
Here, $U(\tzl_k) \in \C^{n\times (2n-1)}$ is a matrix whose columns  form an orthonormal basis for the subspace $S(\tzl_k):=\dkh{\vw\in \C^n:  \Im(\vw^* \tzl_k)=0}$, and
\begin{equation}
\Hl(\tzl_k)=\left [ \begin{array}{l} U(\tzl_k) \vspace{0.4em} \\ \overline{ U(\tzl_k) } \end{array}\right]^* (\Al(\tzl_k))^* \Al(\tzl_k) \left [ \begin{array}{l} U(\tzl_k) \vspace{0.4em} \\ \overline{ U(\tzl_k) } \end{array}\right] \in \R^{(2n-1)\times (2n-1)}
\end{equation}
with $\Al(\tzl_k)$ and $\Fl(\tzl_k)$ are given in \eqref{eq:Azkl} and \eqref{eq:Fzkl}, respectively.  Observe that for any $(2n-1)\times (2n-1)$ real orthogonal matrix $Q$,  the columns of $U(\tzl_k) Q$ also form an orthonormal basis for the subspace $S(\tzl_k)$. Therefore, without loss of generality, we assume 
\begin{equation} \label{eq:UzUlQ}
\norm{\left [ \begin{array}{l} U(\vz_k) \vspace{0.4em} \\ \overline{ U(\vz_k) } \end{array}\right]- \left [ \begin{array}{l} U(\vzl_k) \vspace{0.4em} \\ \overline{ U(\vzl_k) } \end{array}\right]}=\mbox{argmin}_{Q \in \mathcal{O}_{2n-1}} \norm{\left [ \begin{array}{l} U(\vz_k) \vspace{0.4em} \\ \overline{ U(\vz_k) } \end{array}\right]- \left [ \begin{array}{l} U(\vzl_k) \vspace{0.4em} \\ \overline{ U(\vzl_k) } \end{array}\right] Q}.
\end{equation}
For any vector $\vz$, for convenience, we set
\[
\hH(\vz)= \left [ \begin{array}{l} U(\vz) \vspace{0.4em} \\ \overline{ U(\vz) } \end{array}\right] H^{-1}(\vz) \left [ \begin{array}{l} U(\vz) \vspace{0.4em} \\ \overline{ U(\vz) } \end{array}\right]^*
\]
and 
\[
\hH^{(l)}(\vz)= \left [ \begin{array}{l} U(\vz) \vspace{0.4em} \\ \overline{ U(\vz) } \end{array}\right] (H^{(l)}(\vz))^{-1} \left [ \begin{array}{l} U(\vz) \vspace{0.4em} \\ \overline{ U(\vz) } \end{array}\right]^*.
\]
Putting \eqref{eq:upzkp1} and \eqref{eq:upzklp1} into \eqref{eq:diszklp1}, we have
\begin{eqnarray} 
&&  \sqrt2\cdot \norm{\vz_{k+1}- \vzl_{k+1} e^{i \phi_{\vz_k, \vzl_k}^{\mutual}}}  \label{eq:zkl00}\\
&=&\norm{ \left [ \begin{array}{l} \vz_k-\tzl_k \vspace{0.4em} \\ \overline{ \vz_k-\tzl_k } \end{array}\right]  - \hH(\vz_k) A^*(\vz_k)  F(\vz_k)+  \hHl(\tzl_k) (\Al(\tzl_k))^*  \Fl(\tzl_k)} \notag\\
&\le & \underbrace{\norm{ \left [ \begin{array}{l} \vz_k-\tzl_k \vspace{0.4em} \\ \overline{ \vz_k-\tzl_k } \end{array}\right]  - \hH(\vz_k) A^*(\vz_k)  \Big(F(\vz_k)- F(\tzl_k)\Big) }}_{:=I_1}  \notag\\
&& +\underbrace{\norm{\hH(\vz_k) A^*(\vz_k)  F(\tzl_k)- \hH(\vzl_k)   A^*(\tzl_k)  F(\tzl_k)  } }_{:=I_2} \notag\\
&& +\underbrace{\norm{\hH(\vzl_k)   A^*(\tzl_k)  F(\tzl_k) - \hHl(\tzl_k) (\Al(\tzl_k))^*  \Fl(\tzl_k) }}_{:=I_3}. \notag
\end{eqnarray}
We will apply different strategies when upper bounding the terms $I_1, I_2$ and $I_2$, with their bounds given in the following three lemmas under hypotheses \eqref{eq:hypoind}.

\begin{lemma}  \label{le:zkl01}
Under the conditions in Lemma \ref{le:diszkl}, with probability at least $1-O(m^{-10})$, one has
\begin{equation*} 
I_1 \le  \frac{\sqrt2}{10} \norm{\vz_k-\tzl_k},
\end{equation*}
provided $m\ge C_0 n\log^3 m$ for some universal constant $C_0>0$.
\end{lemma}

\begin{lemma} \label{le:zkl02}
Under the conditions in Lemma \ref{le:diszkl}, with probability at least $1-O(m^{-10})$, one has
\begin{equation*}  
I_2 \le \xkh{ 536 \delta+ O\xkh{\sqrt{\frac{\log m} n} } +O\xkh{\sqrt{\frac{n\log^3m}{m}} }+ O\xkh{\sqrt[4 ]{\frac{n\log m}{m} }} \xkh{\lambda+\sqrt{\log m}}}\norm{\vz_k-\tzl_k} ,
\end{equation*}
provided $m\ge C_0 \xkh{\lambda^2+\log^2 m}n\log m$ for some universal constant $C_0>0$.
Here, $\delta$ is the constant given in hypothesis \eqref{eq:hypoinda}, and $\lambda$ is defined in \eqref{eq:lambda}.
\end{lemma}

\begin{lemma} \label{le:zkl03}
Under the conditions in Lemma \ref{le:diszkl}, with probability at least $1-O(m^{-10})-O(me^{-1.5n})$, one has
\begin{equation*} 
I_3 \le O\xkh{\frac{\sqrt{n \log^3m}}{m}},
\end{equation*}
provided $m\ge C_0 n\log^3 m$ for some universal constant $C_0>0$.
\end{lemma}

Putting the bounds in Lemma \ref{le:zkl01}, Lemma \ref{le:zkl02} and Lemma \ref{le:zkl03} into \eqref{eq:zkl00}, together with the definition of $\lambda$ in \eqref{eq:lambda}, one has
\begin{eqnarray*}
\frac{\lambda}{\sqrt m} &=& \dist(\vz_{k+1}, \tzl_{k+1}) \\
&\le &  \xkh{0.1+380 \delta+ O\xkh{ \sqrt[4]{\frac{n\log m}{m}} } \lambda+O\xkh{ \sqrt[4]{\frac{n\log^3m}{m}}}} \norm{\vz_k-\tzl_k} + O\xkh{\frac{\sqrt{n \log^3m}}{m}}.
\end{eqnarray*}
Taking $\delta \le 1/500$, it immediately gives
\begin{eqnarray*}
\lambda &\le  & \frac{\sqrt m \xkh{0.1+380 \delta+ O\xkh{ \sqrt[4]{\frac{n\log^3m}{m}}}} \norm{\vz_k-\tzl_k} + \sqrt n \cdot O\xkh{\frac{\sqrt{n \log^3m}}{m}}}{1-\sqrt m \cdot O\xkh{ \sqrt[4]{\frac{n\log m}{m}} }\cdot \norm{\vz_k-\tzl_k}} \\
&\le & C_2 \sqrt{\log m},
\end{eqnarray*}
provided $m\ge C_0 n\log^3 m$ for a sufficiently large constant $C_0>0$. Here, the last inequality arises from the hypothesis \eqref{eq:hypoindb} that $\norm{\vz_k-\tzl_k} \le C_2\sqrt{\frac{\log m}{m}}$.
This completes the proof.

\end{proof}

%In the sequel, we control these three terms separately as follows.
%\vspace{3em}
\subsection{Proof of Lemma \ref{le:zkl01}}
\hfill\\

 In terms of $I_1$, by the definition of $\tzl_k$, we have $\Im((\vz_k-\tzl_k)^* \vz_k)=0$.  Recall that $U(\vz_k) \in \C^{n\times (2n-1)}$ is a matrix whose columns forms an orthonormal basis for the subspace $S(\vz_k):=\dkh{\vw\in \C^n:  \Im(\vw^* \vz_k)=0}$.
Therefore, 
\begin{equation} \label{eq:zkl11}
\left [ \begin{array}{l} \vz_k-\tzl_k \vspace{0.4em} \\ \overline{\vz_k-\tzl_k } \end{array}\right]=\frac12  \left [ \begin{array}{l} U(\vz_k) \vspace{0.4em} \\ \overline{ U(\vz_k) } \end{array}\right] \left [ \begin{array}{l} U(\vz_k) \vspace{0.4em} \\ \overline{ U(\vz_k) } \end{array}\right]^* \left [ \begin{array}{l} \vz_k-\tzl_k \vspace{0.4em} \\ \overline{ \vz_k-\tzl_k } \end{array}\right].
\end{equation}
Similarly, one can easily verify that
\begin{equation} \label{eq:idenAk1}
A^*(\vz_k) = \frac12  \left [ \begin{array}{l} U(\vz_k) \vspace{0.4em} \\ \overline{ U(\vz_k) } \end{array}\right] \left [ \begin{array}{l} U(\vz_k) \vspace{0.4em} \\ \overline{ U(\vz_k) } \end{array}\right]^*  A^*(\vz_k).
\end{equation}
It then gives
\begin{eqnarray}
I_1& = & \norm{ \hH(\vz_k) A^*(\vz_k)   \xkh{ A(\vz_k)  \left [ \begin{array}{l} \vz_k-\tzl_k \vspace{0.4em} \\ \overline{\vz_k-\tzl_k } \end{array}\right] -  \Big(F(\vz_k)- F(\tzl_k)\Big) } } \notag\\
&\le & \norm{ \hH(\vz_k)  } \norm{A^*(\vz_k)  } \norm{ A(\vz_k)   \left [ \begin{array}{l} \vz_k-\tzl_k \vspace{0.4em} \\ \overline{\vz_k-\tzl_k } \end{array}\right] -  \Big(F(\vz_k)- F(\tzl_k)\Big) } \notag \\
&=&  \norm{ \hH(\vz_k)  } \norm{A^*(\vz_k)  } \norm{  \int_0^1 \Big(A(\vz_k)- A(\vz_\tau)  \Big) ~d\tau \cdot \left [ \begin{array}{l} \vz_k-\tzl_k \vspace{0.4em} \\ \overline{\vz_k-\tzl_k } \end{array}\right] } \label{eq:up31},
\end{eqnarray}
where the first equality comes from \eqref{eq:zkl11} and the identity 
\[
H^{-1}(\vz_k) \left [ \begin{array}{l} U(\vz_k) \vspace{0.4em} \\ \overline{ U(\vz_k) } \end{array}\right]^* A^*(\vz_k) A(\vz_k) \left [ \begin{array}{l} U(\vz_k) \vspace{0.4em} \\ \overline{ U(\vz_k) } \end{array}\right]= H^{-1}(\vz_k) H(\vz_k) = I_{2n-1},
\]
and the last equality follows from the fundamental theorem of calculus. Here,  we denote $\vz_\tau=\tzl_k + \tau(\vz_k- \tzl_k ), 0\le \tau\le 1$. Recall the fact \eqref{eq:hypaftb} that $\max_{1\le j\le m} ~\abs{\va_j^* \vz_k} \lesssim \sqrt{\log m} \norm{\vz_k}$. 
 Applying Lemma \ref{le:Hzlowup}, with probability at least $1-O(m^{-10})$, we have
\begin{equation} \label{eq:Hinvaz1}
\norm{\hH(\vz_k)} \le 2\norm{ H^{-1}(\vz_k)  }  \le  \frac{20}{19\norm{\vz_k}^2}  \qquad \mbox{and} \qquad        \norm{A^*(\vz_k)  } \le \sqrt5 \norm{\vz_k}.
\end{equation}
 Furthermore, we have
\begin{eqnarray}
\max_{1\le j\le m} ~\abs{\va_j^* (\vz_k-\vz_\tau)}&=& (1-\tau) \max_{1\le j\le m} ~\abs{\va_j^* (\vz_k-\tzl_k)} \notag \\
&\le &  (1-\tau)  \cdot \max_{1\le j\le m} ~\norm{\va_j} \norm{\vz_k-\tzl_k} \notag \\
&\le&   \sqrt {6n} \cdot C_2 \sqrt{\frac{\log m} m} \notag \\
&\lesssim & \sqrt { \log m} \label{eq:ajztau}
\end{eqnarray}
for all $0\le \tau \le 1$. Here, the second inequality comes from the hypothesis  \eqref{eq:hypoindb}. Similarly, one can verify that for all $0\le \tau \le 1$, it holds
\begin{equation} \label{eq:ajztau1}
 \norm{\vz_k-\vz_\tau}\le \norm{\vz_k-\tzl_k}  \le  C_2 \sqrt{\frac{\log m} n} \le c_1,
\end{equation}
where $c_1>0$ is a sufficient small constant. Armed with the bounds \eqref{eq:ajztau} and \eqref{eq:ajztau1}, using the same argument as \eqref{eq:LipAz}, one can show that with probability at least $1-O(m^{-10})$, it holds
\begin{equation} \label{eq:LipAz1}
\norm{A(\vz_k)-A(\vz_\tau) } \le 2(1-\tau)\cdot \norm{\vz_k-\tzl_k}+ \frac{1}{25},
\end{equation}
provided $m\ge C_0 n\log^3m$ for some universal constant $C_0>0$.  Putting \eqref{eq:Hinvaz1} and \eqref{eq:LipAz1} into \eqref{eq:up31}, we obtain
\[
I_1 \le  \frac{20\sqrt5}{19\norm{\vz_k}}  \xkh{2(1-\tau)\cdot \norm{\vz_k-\tzl_k}+ \frac{1}{25}}
\cdot \sqrt2 \norm{\vz_k-\tzl_k}\le \frac{\sqrt2}{10} \norm{\vz_k-\tzl_k}.
\]
Here, we use the fact \eqref{eq:hypafta} that $ 0.99 \le \norm{\vz_k} \le 1.01$ and the hypothesis \eqref{eq:hypoindb}. This completes the proof.

$\hfill\square$

\subsection{Proof of Lemma \ref{le:zkl02}}
\hfill\\

For the term $I_2$, we first introduce the notation
\begin{equation} \label{eq:atilde}
\tA(\vz)=A(\vz)  \left [ \begin{array}{l} U(\vz) \vspace{0.4em} \\ \overline{ U(\vz) } \end{array}\right] \in \R^{m\times (2n-1)}
\end{equation}
for any $\vz$. From the hypothesis \eqref{eq:hypaftb}, one has $\max_{1\le j\le m} ~\abs{\va_j^* \vz_k} \lesssim \sqrt{\log m} \norm{\vz_k}$. Therefore, Lemma \ref{le:Hzlowup} implies that $\tA(\vz_k)^*\tA(\vz_k)$ is injective, and so  $ \tA(\vz_k)^{\dag}=\xkh{\tA(\vz_k)^*\tA(\vz_k)}^{-1} \tA(\vz_k)^*$.
Similarly, using the fact \eqref{eq:hypaftd}, we have $ \tA(\tzl_k)^{\dag}=\xkh{\tA(\tzl_k)^*\tA(\tzl_k)}^{-1} \tA(\tzl_k)^*$. With the above relation, one sees
\begin{eqnarray*}
I_2 & = & \norm{ \left [ \begin{array}{l} U(\vz_k) \vspace{0.4em} \\ \overline{ U(\vz_k) } \end{array}\right]  \tA(\vz_k)^{\dag} F(\tzl_k) - \left [ \begin{array}{l} U(\tzl_k) \vspace{0.4em} \\ \overline{ U(\tzl_k) } \end{array}\right]  \tA(\tzl_k)^{\dag} F(\tzl_k)    } \\
&\le & \underbrace{\norms{\left [ \begin{array}{l} U(\vz_k) \vspace{0.4em} \\ \overline{ U(\vz_k) } \end{array}\right] - \left [ \begin{array}{l} U(\tzl_k) \vspace{0.4em} \\ \overline{ U(\tzl_k) } \end{array}\right] }_{2,\Re} \norm{\tA(\vz_k)^{\dag} } \norm{F(\tzl_k) }}_{:=\beta_1} \\
&& + \norms{  \left [ \begin{array}{l} U(\tzl_k) \vspace{0.4em} \\ \overline{ U(\tzl_k) } \end{array}\right]}_{2,\Re} \underbrace{\norm{\tA(\vz_k)^{\dag} F(\tzl_k) - \tA(\tzl_k)^{\dag} F(\tzl_k) }}_{:=\beta_2}
\end{eqnarray*}

1.  For $\beta_1$, applying Lemma \ref{le:subper} with the fact \eqref{eq:UzUlQ},  we have
\begin{equation} \label{eq:bata11}
\norms{\left [ \begin{array}{l} U(\vz_k) \vspace{0.4em} \\ \overline{ U(\vz_k) } \end{array}\right] - \left [ \begin{array}{l} U(\tzl_k) \vspace{0.4em} \\ \overline{ U(\tzl_k) } \end{array}\right] }_{2,\Re} \le \frac{ \sqrt2\norm{\vz-\tzl_k} }{\sqrt{\norm{\vz_k}\norm{\tzl_k}} }.
\end{equation}
For the term $\norm{\tA(\vz_k)^{\dag} }$,  by the definition of $\tA$ as in \eqref{eq:atilde}, it then follows from  Lemma \ref{le:Hzlowup} that
\begin{equation}  \label{eq:bata12}
\norm{\tA(\vz_k)^{\dag} } =\norm{\xkh{\tA(\vz_k)^*\tA(\vz_k)}^{-1} \tA(\vz_k)^*} \le \norm{\xkh{\tA(\vz_k)^*\tA(\vz_k)}^{-1}} \cdot \sqrt2 \norm{A(\vz_k)} \le  \frac{10\sqrt{10}} {19\norm{\vz_k}}.
\end{equation}
For the term $\norm{F(\tzl_k)}$, observe that $F(\vxs e^{i\phi(\vz_{k})})=0$. Therefore,
\begin{equation} \label{eq:Fzklt}
\norm{F(\tzl_k)} =\norm{F(\tzl_k)-F(\vxs e^{i\phi(\vz_{k})})}= \norm{\int_0^1 A(\vz_\tau)   \left [ \begin{array}{l} \tzl_k- \vxs e^{i\phi(\vz_{k})} \vspace{0.4em} \\ \overline{\tzl_k- \vxs e^{i\phi(\vz_{k})}} \end{array}\right] ~d\tau},
\end{equation}
where $\vz_\tau= \vxs e^{i\phi(\vz_{k})}+\tau(\tzl_k- \vxs e^{i\phi(\vz_{k})})$. Here, the last identity is due to the fundamental theorem of calculus.   Controlling the last term in \eqref{eq:Fzklt} requires the following two consequences
\begin{equation} \label{eq:zkajtau}
\norm{\tzl_k- \vxs e^{i\phi(\vz_{k})}} \le 0.11 \qquad \mbox{and} \qquad \max_{1\le j \le m} ~\abs{\va_j^* \vz_\tau} \lesssim  \sqrt{\log m}  \norm{\vz_\tau}.
\end{equation}
To see the left statement in \eqref{eq:zkajtau}, one has
\begin{equation} \label{eq:zkldix}
\norm{\tzl_k- \vxs e^{i\phi(\vz_{k})}}  \le  \norm{\vz_k-\tzl_k} +\norm{\vz_k- \vxs e^{i\phi(\vz_{k})}} \le C_2 \sqrt{\frac{\log m} m} + \delta \le 0.11,
\end{equation}
where the second inequality follows from \eqref{eq:hypoinda} and \eqref{eq:hypoindb}. Moreover, for the right statement in \eqref{eq:zkajtau}, one can see 
\begin{eqnarray*}
\max_{1\le j \le m} ~\abs{\va_j^* \vz_\tau} &= & \tau \max_{1\le j \le m} ~ \abs{\va_j^* \tzl_k} +(1-\tau) \max_{1\le j \le m} ~ \abs{\va_j^* \vxs e^{i\phi(\vz_{k})}} \\
&  \lesssim & \tau \sqrt{\log m} \norm{\tzl_k} + (1-\tau) \sqrt{\log m} \\
&\lesssim &  \sqrt{\log m}  \norm{\vz_\tau},
\end{eqnarray*}
where the first inequality comes from \eqref{eq:hypaftd} and \eqref{eq:inajxs},  and  the last inequality arises from 
\begin{equation} \label{eq:zkldix22}
0.89 \le  \norm{\vxs}-\norm{\tzl_k- \vxs e^{i\phi(\vz_{k})}}  \le  \norm{\vz_\tau}  \le  \norm{\vxs}+ \norm{\tzl_k- \vxs e^{i\phi(\vz_{k})}} \le 1.11.
\end{equation}
Here, we use the bounds \eqref{eq:zkldix} in the above inequality. Armed with these bounds \eqref{eq:zkajtau} and \eqref{eq:zkldix22}, we can readily apply Lemma \ref{le:Hzlowup} to obtain 
\[
\norm{A(\vz_\tau)} \le \sqrt 5 \norm{\vz_\tau}.
\]
Putting it into \eqref{eq:Fzklt}, we obtain
\begin{eqnarray}
\norm{F(\tzl_k)}  & \le &  \sqrt{10} \int_0^1 \norm{\vz_\tau}~d\tau \norm{\tzl_k- \vxs e^{i\phi(\vz_{k})} } \notag\\
&\le & \sqrt{10} \cdot 1.11 \cdot  \xkh{C_2 \sqrt{\frac{\log m} m} + \delta}  \notag\\
&\le & O\xkh{\sqrt{\frac{\log m} n}} + 3.6 \delta,  \label{eq:bata13}
\end{eqnarray}
where the second inequality arises from \eqref{eq:zkldix} and \eqref{eq:zkldix22}.

Collecting the previous three bounds \eqref{eq:bata11},  \eqref{eq:bata12}, and \eqref{eq:bata13}, one can reach
\[
\beta_1\le \frac{ 20 \sqrt{5} \norm{\vz-\tzl_k} }{ 19\norm{\vz_k}\sqrt{\norm{\vz_k}\norm{\tzl_k}} } \xkh{O\xkh{\sqrt{\frac{\log m} n}} + 3.6 \delta}  \le \xkh{O\xkh{\sqrt{\frac{\log m} n} }+ 10\delta} \norm{\vz-\tzl_k}.
\]
Here,  the last inequality comes from the facts \eqref{eq:hypafta} and \eqref{eq:hypaftc}.

2. For the term $\beta_2$,  it follows from Lemma \ref{le:Hzlowup} together with the fact \eqref{eq:hypaftd} that  the nullspace $N(\tA(\tzl_k))=\bm{0}$. Therefore, by the decomposition theorem for pseudo-inverse (Lemma \ref{le:psudecom}), one has
\begin{eqnarray*}
\beta_2&\le & \norm{\tA(\tzl_k)^{\dag}\xkh{\tA(\tzl_k)-\tA(\vz_k)}\tA(\vz_k)^{\dag} F(\tzl_k) } \\
&&+ \norm{\xkh{(\tA(\tzl_k))^* \tA(\tzl_k)}^\dag  \xkh{\tA(\tzl_k)-\tA(\vz_k)}^* P_{N(\tA^*(\vz_k))} F(\tzl_k) } \\
&\le &\norm{\tA(\tzl_k)^{\dag}}  \underbrace{ \norm{\xkh{\tA(\tzl_k)-\tA(\vz_k)}\tA(\vz_k)^{\dag} F(\tzl_k)  }}_{:=\beta_{21}} \\
&&+\norm{\tA(\tzl_k)^\dag}^2  \underbrace{  \norm{\xkh{\tA(\tzl_k)-\tA(\vz_k)}^*\tA(\vz_k)\tA(\vz_k)^\dag F(\tzl_k)}}_{:=\beta_{22}} \\
&& +\norm{\tA(\tzl_k)^\dag}^2  \underbrace{ \norm{\xkh{\tA(\tzl_k)-\tA(\vz_k)}^*F(\tzl_k)}}_{:=\beta_{23}} 
\end{eqnarray*}
where the last inequality comes from the identity that $P_{N(\tA^*(\vz_k))}=I-\tA(\vz_k)\tA(\vz_k)^\dag $. As the same argument to \eqref{eq:bata12}, we have
\begin{eqnarray}  \label{eq:bata2inv}
\norm{\tA(\tzl_k)^{\dag} } & = & \norm{\xkh{\tA(\tzl_k)^*\tA(\tzl_k)}^{-1} \tA(\tzl_k)^*} \\
&\le&  \norm{\xkh{\tA(\tzl_k)^*\tA(\tzl_k)}^{-1}} \cdot \sqrt 2 \norm{A(\tzl_k)}  \notag\\
&\le&   \frac{10\sqrt {10}} {19 \norm{\tzl_k}}. \notag
\end{eqnarray}

2.1.  For the term $\beta_{21}$, let
\[
\vw:=\tA(\vz_k)^{\dag} F(\tzl_k).
\]
A direct consequence of  \eqref{eq:bata12} and \eqref{eq:bata13} is that 
\begin{equation} \label{eq:normvw}
\norm{\vw} \le \frac{10\sqrt {10}} {19 \norm{\tzl_k}} \xkh{O\xkh{\sqrt{\frac{\log m} n}}+ 3.6 \delta} \le O\xkh{\sqrt{\frac{\log m} n} } +6.1 \delta.
\end{equation}
Here, the last inequality comes from the fact \eqref{eq:hypafta}.
Recall that 
\begin{equation*} 
\tA(\vz)=A(\vz)  \left [ \begin{array}{l} U(\vz) \vspace{0.4em} \\ \overline{ U(\vz) } \end{array}\right] \in \R^{m\times (2n-1)}
\end{equation*}
Therefore,
\begin{eqnarray*}
\beta_{21} &= & \norm{\xkh{A(\tzl_k)\left [ \begin{array}{l} U(\tzl_k) \vspace{0.4em} \\ \overline{ U(\tzl_k) } \end{array}\right] - A(\vz_k) \left [ \begin{array}{l} U(\vz_k) \vspace{0.4em} \\ \overline{ U(\vz_k) } \end{array}\right] }\vw } \\
&\le &  \underbrace{ \norm{A(\tzl_k) }\norms{ \left [ \begin{array}{l} U(\tzl_k) \vspace{0.4em} \\ \overline{ U(\tzl_k) } \end{array}\right] -  \left [ \begin{array}{l} U(\vz_k) \vspace{0.4em} \\ \overline{ U(\vz_k) } \end{array}\right] }_{2,\Re} \norm{\vw}}_{:=r_1} \\
&&+ \underbrace{\norm{\xkh{A(\tzl_k)-A(\vz_k)} \left [ \begin{array}{l} U(\vz_k) \vspace{0.4em} \\ \overline{ U(\vz_k) } \end{array}\right] \vw }}_{:=r_2}.
\end{eqnarray*}
The term $r_1$ is relatively simple to control. Using Lemma \ref{le:Hzlowup} and Lemma \ref{le:subper}, together with \eqref{eq:normvw}, one has
\begin{eqnarray*}
r_1 & \le &  \sqrt 5 \norm{\tzl_k} \cdot  \frac{\sqrt 2\norm{\vz-\tzl_k} }{\sqrt{\norm{\vz_k}\norm{\tzl_k}} } \cdot \xkh{O\xkh{\sqrt{\frac{\log m} n} }+ 6.1 \delta} \\
& \le &  \xkh{O\xkh{\sqrt{\frac{\log m} n}} + 23\delta} \norm{\vz-\tzl_k}.
\end{eqnarray*}
Here,  the last inequality comes from the facts \eqref{eq:hypafta} and \eqref{eq:hypaftc}, and $C_5>0$ is a universal constant.

Moving on to the term $r_2$, observe that 
\begin{eqnarray} \label{eq:vv}
\vv:=U(\vz_k)  \vw & = & U(\vz_k)   \tA(\vz_k)^{\dag} F(\vz_k) +U(\vz_k)  \tA(\vz_k)^{\dag} (F(\tzl_k) -F(\vz_k)) \notag \\
&=& \vz_k -\vz_{k+1} +U(\vz_k)  \tA(\vz_k)^{\dag} (F(\tzl_k) -F(\vz_k)),
\end{eqnarray}
where we use the Gauss-Newton update rule $\vz_{k+1}=\vz_k- U(\vz_k)   \tA(\vz_k)^{\dag} F(\vz_k) $ in the second equation. We next show that with probability exceeding $1-O(m^{-10})-O(me^{-1.5n})$, it holds
\begin{equation}  \label{eq:ajv0}
\max_{1\le j \le m} ~\abs{\va_j^* \vv}  \lesssim   \lambda+\sqrt{\log m},
\end{equation}
where  $\lambda$ is defined in \eqref{eq:lambda}.   To see this, one has
\begin{equation} \label{eq:ajv}
\max_{1\le j \le m} ~\abs{\va_j^* \vv} \le  \max_{1\le j \le m} ~ \abs{\va_j^* \vz_k} +  \max_{1\le j \le m} ~ \abs{\va_j^* \vz_{k+1}} + \max_{1\le j \le m} ~ \norm{\va_j} \cdot  \norm{  \tA(\vz_k)^{\dag} } \norm{F(\tzl_k) -F(\vz_k)}.
\end{equation}
In view of \eqref{eq:hypafta} and  \eqref{eq:hypaftb}, we have
\begin{equation}  \label{eq:ajv1}
 \max_{1\le j \le m} ~ \abs{\va_j^* \vz_k} \lesssim \sqrt{\log m}.
\end{equation}
Furthermore, due to the independence between $\va_l$ and $\vzl_{k+1}$,  one can apply standard Gaussian concentration inequalities to show that with probability at least $1-O(m^{-10})$ 
\begin{eqnarray*}
\max_{1\le l \le m} ~ \abs{\va_l^* \tzl_{k+1}} = \max_{1\le l \le m} ~ \abs{\va_l^* \vzl_{k+1}} &  \le &  5 \sqrt{\log m} \norm{\vzl_{k+1}} \\
&  \overset{\text{(i)}}{\le} &  5 \sqrt{\log m} \xkh{ \norm{\vz_{k+1}}+\dist(\vz_{k+1},\tzl_{k+1}) } \\ 
&\overset{\text{(ii)}}{\le}  & 5 \sqrt{\log m} \xkh{1.1+\frac{\lambda}{\sqrt m} } \\
& \lesssim&  \sqrt{\log m} +\lambda,
\end{eqnarray*} 
where (i) arises from the fact $\dist(\vz_{k+1}, \vxs) \le \dist(\vz_{k}, \vxs)  \le \delta \le 0.1$ by Lemma \ref{le:conun1} and \eqref{eq:hypoinda} and \eqref{eq:hypoindc}, and (ii) comes from the definition of $\lambda$ given in \eqref{eq:lambda}. Therefore, 
\begin{eqnarray}  
 \max_{1\le l \le m} ~ \abs{\va_l^* \vz_{k+1}} &  \le  &  \max_{1\le l \le m} ~ \abs{\va_l^* \tzl_{k+1}}  +  \max_{1\le l \le m} ~ \Abs{\va_l^* \xkh{ \vz_{k+1} -\tzl_{k+1} }}  \notag \\
 &\lesssim  &  \sqrt{\log m} +\lambda +   \max_{1\le l \le m} ~ \norm{\va_l}  \max_{1\le l \le m} ~\norm{\vz_{k+1} -\tzl_{k+1} } \notag\\
 &\lesssim& \sqrt{\log m} +\lambda  +\sqrt{6n}  \cdot \frac{\lambda}{\sqrt m}\notag \\
  &\lesssim& \sqrt{\log m} +\lambda.  \label{eq:ajv2}
\end{eqnarray}
For the last term of \eqref{eq:ajv}, one has
\[
 \norm{F(\tzl_k) -F(\vz_k)} =\norm{\int_0^1 A(\vz_k+\tau(\tzl_k-\vz_k))~d\tau \cdot (\vz_k-\tzl_k) } \le C_6 \norm{\vz_k-\tzl_k}.
\]
Here, the inequality follows from Lemma \ref{le:Hzlowup} and the facts \eqref{eq:hypaft}, and $C_6>0$ is a universal constant.  It immediately gives 
\begin{eqnarray}
 \max_{1\le j \le m} ~ \norm{\va_j} \cdot  \norm{  \tA(\vz_k)^{\dag} } \norm{F(\tzl_k) -F(\vz_k)}  & \le  & \sqrt{6n} \cdot \frac{\sqrt 5}{\norm{\vz_k}} \cdot  C_6\norm{\vz_k-\tzl_k} \notag\\
 &\le  &C_6  \frac{\sqrt{30 n}}{\norm{\vz_k}} \cdot C_2 \sqrt{\frac{\log m}{m}} \notag \\
 &\lesssim & \sqrt{\log m}.  \label{eq:ajv3}
\end{eqnarray}
Here,  the first inequality comes from \eqref{eq:maallaj} and \eqref{eq:bata11},  the second inequality  follows from the hypothesis \eqref{eq:hypoindb}, and the last inequality arises from the fact \eqref{eq:hypafta}. Putting \eqref{eq:ajv1}, \eqref{eq:ajv2} and \eqref{eq:ajv3} into \eqref{eq:ajv}, we completes the proof of \eqref{eq:ajv0}.

Armed with the bound \eqref{eq:ajv0}, we can readily apply Lemma \ref{le:contra2} to obtain that with probability at least $1-O(m^{-10})$, it holds
\begin{eqnarray*}
r_2 &\le & 2 \sqrt{ \frac{1}{m} \sum_{j=1}^m \abs{\va_j^* \vv}^2 \abs{\va_j^*(\vz-\tzl_k)}^2 } \\
 &\le & 2\norm{\vz_k-\tzl_k} \cdot \sqrt{2\norm{\vv}^2+O\xkh{\sqrt{\frac{n\log m}{m} }} \xkh{\lambda+\sqrt{\log m}}^2} \\
 &\le &   \norm{\vz_k-\tzl_k} \cdot \xkh{ 13 \delta+ O\xkh{\sqrt{\frac{\log m} n} } + O\xkh{\sqrt[4 ]{\frac{n\log m}{m} }} \xkh{\lambda+\sqrt{\log m}}},
\end{eqnarray*}
provided $m\ge C_0\xkh{\lambda^2+\log^2m}n\log m$. Here, the last inequality comes from \eqref{eq:normvw} and the fact $\sqrt2\norm{\vv} =\norm{\vw}$ by \eqref{eq:vv}. Combining the estimators for $r_1$ and $r_2$, one has
\[
\beta_{21} \le   \norm{\vz_k-\tzl_k} \cdot \xkh{ 36 \delta+ O\xkh{\sqrt{\frac{\log m} n} } + O\xkh{\sqrt[4 ]{\frac{n\log m}{m} }} \xkh{\lambda+\sqrt{\log m}}}.
\]

2.2. For the term $\beta_{22}$, using the same notations as when estimating $\beta_{21}$, we have
\begin{eqnarray*}
\beta_{22} &=&   \norm{ \xkh{ \left [ \begin{array}{l} U(\tzl_k) \vspace{0.4em} \\ \overline{ U(\tzl_k) } \end{array}\right]^*  A^*(\tzl_k)- \left [ \begin{array}{l} U(\vz_k) \vspace{0.4em} \\ \overline{ U(\vz_k) } \end{array}\right]^* A^*(\vz_k) } A(\vz_k) \left [ \begin{array}{l} \vv \vspace{0.4em} \\ \overline{ \vv } \end{array}\right]   }  \\
&\le & \underbrace{\sqrt2 \norms{ \left [ \begin{array}{l} U(\tzl_k) \vspace{0.4em} \\ \overline{ U(\tzl_k) } \end{array}\right] -  \left [ \begin{array}{l} U(\vz_k) \vspace{0.4em} \\ \overline{ U(\vz_k) } \end{array}\right] }_{2,\Re} \norm{A^*(\tzl_k) } \norm{A(\vz_k)} \norm{\vv}}_{:=\omega_1}  \\
&& + \sqrt2 \underbrace{\norm{ \xkh{A(\tzl_k)-A(\vz_k)}^* A(\vz_k) \left [ \begin{array}{l} \vv \vspace{0.4em} \\ \overline{ \vv } \end{array}\right]  } }_{:=\omega_2}.
\end{eqnarray*}
Here, $\vv$ is given in \eqref{eq:vv}.  Using the same argument as when estimating $r_1$ above,  one has
\begin{eqnarray*}
\omega_1 & \le & \sqrt 2 \cdot  \frac{ \sqrt2\norm{\vz-\tzl_k} }{\sqrt{\norm{\vz_k}\norm{\tzl_k}} } \cdot \sqrt 5 \norm{\tzl_k} \cdot \sqrt 5 \norm{\vz_k}  \cdot \frac{1} {\sqrt2} \xkh{O\xkh{\sqrt{\frac{\log m} n} }+ 6.1\delta} \\
& \le &  \xkh{O\xkh{\sqrt{\frac{\log m} n}} + 52\delta} \norm{\vz-\tzl_k}.
\end{eqnarray*}
Here, we use the facts \eqref{eq:hypaft} in the last inequality.  For the term $\omega_2$, by the definition of $A(\vz)$ given in \eqref{eq:Azk},  one has 
\begin{eqnarray*}
\omega_2 & \le &  2 \xkh{\norm{ \frac1m \sum_{j=1}^m  \vz_k^* \va_j \va_j^* \vv \va_j\va_j^* (\vz-\tzl_k)} +\norm{ \frac1m \sum_{j=1}^m  \vv^* \va_j \va_j^* \vz_k   \va_j\va_j^* (\vz-\tzl_k)} }  \\
&\le & 4   \norm{ \frac1m \sum_{j=1}^m  \vz_k^* \va_j \va_j^* \vv \va_j\va_j^* }  \norm{\vz-\tzl_k}.
\end{eqnarray*}
Note that
\begin{eqnarray*}
 \norm{ \frac1m \sum_{j=1}^m  \vz_k^* \va_j \va_j^* \vv \va_j\va_j^* } & \overset{\text{(i)}}{\le} &  \sqrt{  \norm{  \frac1m \sum_{j=1}^m \abs{\va_j^* \vz_k}^2 \va_j\va_j^*}}  \sqrt{  \norm{  \frac1m \sum_{j=1}^m \abs{\va_j^* \vv}^2 \va_j\va_j^*}} \\
 &\overset{\text{(ii)}}{\le} & \sqrt{2\norm{\vz_k}^2 + O\xkh{ \sqrt{\frac{n\log^3m}{m}}} } \sqrt{2\norm{\vv}^2+O\xkh{ \sqrt{\frac{n\log m}{m}}}\xkh{ \lambda+\sqrt{\log m} }^2  } \\
 &\overset{\text{(iii)}}{\le} & 9\delta+ O\xkh{ \sqrt[4]{\frac{n\log m}{m}}}\xkh{ \lambda+\sqrt{\log m} } + O\xkh{\sqrt{\frac{\log m}{n}}},
\end{eqnarray*} 
provided $m\ge C_0 n\log ^3 m$ for some universal constant $C_0>0$.
Here, (i) comes from Cauchy-Schwarz inequality, (ii) arises from Lemma \ref{le:Hzlowup} and the facts \eqref{eq:hypaftb}, \eqref{eq:ajv0}, and (iii) comes from \eqref{eq:hypafta}, \eqref{eq:normvw} and \eqref{eq:vv}. The previous three bounds taken collectively yield 
\[
\beta_{22} \le \xkh{88\delta+ O\xkh{ \sqrt[4]{\frac{n\log m}{m}}}\xkh{ \lambda+\sqrt{\log m} } + O\xkh{\sqrt{\frac{\log m}{n}}}}\cdot  \norm{\vz_k-\tzl_k}.
\]

2.3.  For the term $\beta_{23}$, one can apply the triangle inequality to get
\begin{eqnarray*}
\beta_{23} &=& \norm{ \xkh{ \left [ \begin{array}{l} U(\tzl_k) \vspace{0.4em} \\ \overline{ U(\tzl_k) } \end{array}\right]^*  A^*(\tzl_k)- \left [ \begin{array}{l} U(\vz_k) \vspace{0.4em} \\ \overline{ U(\vz_k) } \end{array}\right]^* A^*(\vz_k) } F(\tzl_k)   } \\
&\le & \underbrace{ \norm{A(\tzl_k) }\norms{ \left [ \begin{array}{l} U(\tzl_k) \vspace{0.4em} \\ \overline{ U(\tzl_k) } \end{array}\right] -  \left [ \begin{array}{l} U(\vz_k) \vspace{0.4em} \\ \overline{ U(\vz_k) } \end{array}\right] }_{2,\Re} \norm{F(\tzl_k)  }}_{:=\theta_1} \\
&& + \underbrace{\sqrt2 \norm{ \xkh{A(\tzl_k)-A(\vz_k)}^* F(\tzl_k)  } }_{:=\theta_2}. \\
\end{eqnarray*}
By utilizing Lemma \ref{le:Hzlowup} and Lemma \ref{le:subper}, along with the facts \eqref{eq:hypaft} and \eqref{eq:bata13}, we obtain
\begin{eqnarray} \label{eq:theta1}
\theta_1 &\le&  \sqrt5 \norm{\tzl_k} \cdot \frac{ \sqrt2\norm{\vz-\tzl_k} }{\sqrt{\norm{\vz_k}\norm{\tzl_k}} } \cdot \xkh{O\xkh{\sqrt{\frac{\log m} n}}+3.6\delta}  \notag \\
& \le &   \xkh{O\xkh{\sqrt{\frac{\log m} n}} + 13\delta} \norm{\vz-\tzl_k}.
\end{eqnarray}
Recall the definitions of $A(\vz)$ and $F(\vz)$ given in \eqref{eq:Azk} and \eqref{eq:Fzk}, respectively. One has
\begin{eqnarray} \label{eq:theta2}
\theta_2 &=& 2 \cdot  \norm{ \frac1m \sum_{j=1}^m \xkh{\abs{\va_j^* \tzl_k}^2- \abs{\va_j^* \vxs}^2 } \va_j\va_j^* (\vz-\tzl_k)} \notag \\
&\le &  2  \norm{\vz-\tzl_k} \cdot  \norm{ \frac1m \sum_{j=1}^m \xkh{\abs{\va_j^* \tzl_k}^2- \abs{\va_j^* \vxs}^2 } \va_j\va_j^*}. 
\end{eqnarray}
Combine Lemma \ref{le:contra} and the facts \eqref{eq:hypaft} and \eqref{eq:inajxs} to see that
\begin{eqnarray*}
&& \norm{ \frac1m \sum_{j=1}^m \xkh{\abs{\va_j^* \tzl_k}^2- \abs{\va_j^* \vxs}^2 } \va_j\va_j^* -  \xkh{  \tzl_k (\tzl_k)^* + \norm{ \tzl_k}^2 I - \vxs(\vxs)^* -\norm{\vxs}^2 I} } \\
&\lesssim & \sqrt{\frac{n\log^3m}{m}} \xkh{\norm{ \tzl_k}^2 +\norm{\vxs}^2} \lesssim   \sqrt{\frac{n\log^3m}{m}}.
\end{eqnarray*}
This further allows one to derive
\begin{eqnarray} \label{eq:theta3}
 && \norm{ \frac1m \sum_{j=1}^m \xkh{\abs{\va_j^* \tzl_k}^2- \abs{\va_j^* \vxs}^2 } \va_j\va_j^*} \\
 &\le &  \norm{ \tzl_k (\tzl_k)^* + \norm{ \tzl_k}^2 I - \vxs(\vxs)^* -\norm{\vxs}^2 I } +O\xkh{\sqrt{\frac{n\log^3m}{m}} } \notag \\
 & \le & 2 \xkh{\norm{\tzl_k}+\norm{\vxs}} \norm{\tzl_k-\vxs e^{i\phi(\vz_{k})}}+O\xkh{\sqrt{\frac{n\log^3m}{m}} } \notag \\
 & \le & 4.3\delta+ O\xkh{\sqrt{\frac{\log m}{n}}}  +O\xkh{\sqrt{\frac{n\log^3m}{m}} },
\end{eqnarray}
where the last inequality arises from the fact  \eqref{eq:zkldix} that 
\[
 \norm{\tzl_k-\vxs e^{i\phi(\vz_{k})}} \le O\xkh{\sqrt{\frac{\log m}{n}}}+\delta.
\]
Putting \eqref{eq:theta1}, \eqref{eq:theta2}, and \eqref{eq:theta3} together, one yields 
\[
\beta_{23} \le \xkh{ 22\delta+ O\xkh{\sqrt{\frac{\log m}{n}}}  +O\xkh{\sqrt{\frac{n\log^3m}{m}} }} \norm{\vz-\tzl_k}.
\]

Collecting $\beta_{21}, \beta_{22}, \beta_{23}$ together with $\beta_1$,  we have 
\[
I_2 \le  \xkh{ 536 \delta+ O\xkh{\sqrt{\frac{\log m} n} } +O\xkh{\sqrt{\frac{n\log^3m}{m}} }+ O\xkh{\sqrt[4 ]{\frac{n\log m}{m} }} \xkh{\lambda+\sqrt{\log m}}}\norm{\vz_k-\tzl_k} .
\]

$\hfill\square$

\subsection{Proof of Lemma \ref{le:zkl03}} \hfill\\

For the term $I_3$, set
\begin{equation} \label{eq:atildel}
\tAl(\tzl_k)= \Al(\vzl_k) \left [ \begin{array}{l} U(\vzl_k) \vspace{0.4em} \\ \overline{ U(\vzl_k) } \end{array}\right].
\end{equation}
For convenience, we use $\tA, \tAl $ short for $\tA(\tzl_k), \tAl(\tzl_k) $, respectively. Then the term $I_3$ can be rewritten as
\begin{eqnarray*}
I_3 &=& \norm{\left [ \begin{array}{l} U(\tzl_k) \vspace{0.4em} \\ \overline{ U(\tzl_k)} \end{array}\right] \xkh{  (\tA^* \tA)^{-1} \tA^*  F(\tzl_k) - ( (\tAl)^*  \tAl)^{-1}  (\tAl)^*  \Fl(\tzl_k)  }}
\end{eqnarray*}
From the definitions of \eqref{eq:Azkl} and \eqref{eq:Fzkl}, we know 
\[
A^*(\vzl_k)A(\vzl_k) = (\Al(\vzl_k))^*\Al(\vzl_k)+ \frac 1m
 \left [ \begin{array}{ll}  \abs{\va_l^* \tzl_k}^2 \va_l\va_l^* &    \xkh{\va_l^*\tzl_k}^2 \va_l \va_l^\T   \\   \xkh{(\tzl_k)^* \va_l}^2 \overline{\va_l} \va_l^* &  \abs{\va_l^* \tzl_k}^2 \overline{\va_l} \va_l^\T   \end{array}\right].
\]
and
\[
A^*(\vzl_k)F(\tzl_k)=\Al(\tzl_k)^*  \Fl(\tzl_k) + \frac1m \xkh{\abs{\va_l^* \tzl_k}^2- \abs{\va_l^* \vxs}^2}  \left [ \begin{array}{l} \va_l \va_l^* \tzl_k \vspace{0.5em} \\ \overline{ \va_l \va_l^* \tzl_k } \end{array}\right]
\]
It then gives
\begin{equation} \label{eq:tAA}
\tA^* \tA= (\tAl)^*  \tAl+ \vu_l \vu_l^*
\end{equation}
and 
\begin{equation} \label{eq:tAFzkl}
 \tA^*  F(\tzl_k)= (\tAl)^*  \Fl(\tzl_k) + \vw_l,
\end{equation}
where 
\[
\vu_l =  \frac1{\sqrt m} \left [ \begin{array}{l} U(\vzl_k) \vspace{0.4em} \\ \overline{ U(\vzl_k) } \end{array}\right]^* \left [ \begin{array}{l} \va_l \va_l^* \tzl_k \vspace{0.5em} \\ \overline{ \va_l \va_l^* \tzl_k } \end{array}\right] \in \R^n
\]
and 
\[
\vw_l=  \frac1m \xkh{\abs{\va_l^* \tzl_k}^2- \abs{\va_l^* \vxs}^2} \left [ \begin{array}{l} U(\vzl_k) \vspace{0.4em} \\ \overline{ U(\vzl_k) } \end{array}\right]^*  \left [ \begin{array}{l} \va_l \va_l^* \tzl_k \vspace{0.5em} \\ \overline{ \va_l \va_l^* \tzl_k } \end{array}\right]  \in \R^{2n-1}.
\]
Using Sherman-Morrison formula (Lemma \ref{le:invranone}) for $(\tA^* \tA)^{-1}$  together with \eqref{eq:tAA}, we have
\begin{equation} \label{eq:taAinv}
(\tA^* \tA)^{-1}= ( (\tAl)^*  \tAl)^{-1}-\frac{ ( (\tAl)^*  \tAl)^{-1} \vu_l \vu_l^*  ( (\tAl)^*  \tAl)^{-1}}{1+\vu_l^*  ( (\tAl)^*  \tAl)^{-1} \vu_l}.
\end{equation}
For convenience, set
\[
\vv^{(l)}:=\xkh{ (\tAl)^*  \tAl}^{-1} (\tAl)^*  \Fl(\tzl_k) \in \R^{2n-1}.
\]
It then follows from \eqref{eq:tAFzkl} and \eqref{eq:taAinv} that
\begin{eqnarray*}
I_3&= &\sqrt2 \norm{  (\tA^* \tA)^{-1} \tA^*  F(\tzl_k) - ( (\tAl)^*  \tAl)^{-1}  (\tAl)^*  \Fl(\tzl_k) } \\
&= &\sqrt2 \norm{ (\tA^* \tA)^{-1} \vw_l - \frac{ \xkh{(\tAl)^*  \tAl}^{-1} \vu_l \vu_l^* \vv^{(l)}}{1+\vu_l^*  ( (\tAl)^*  \tAl)^{-1} \vu_l} }\\
&\overset{\text{(i)}}{\le}& \sqrt2\norm{ (\tA^* \tA)^{-1}} \norm{\vw_l}+ \sqrt2\norm{ \xkh{(\tAl)^*  \tAl}^{-1} } \norm{\vu_l} \Abs{\vu_l^* \vv^{(l)}} \\
&\overset{\text{(ii)}}{\lesssim} & \frac{1}{ \norm{\tzl_k}^2} \cdot  \frac{\sqrt{n \log^3m}}{m} \norm{\tzl_k}^2+ \frac{1}{ \norm{\tzl_k}^2} \cdot \sqrt{\frac{n\log m}{m}} \norm{\tzl_k} \cdot  \frac{\log m }{\sqrt m} \\
&\lesssim& \frac{\sqrt{n \log^3m}}{m},
\end{eqnarray*}
Here, (i) arises from the Cauchy-Schwarz inequality and the fact $\sigma_{\min}\xkh{\xkh{(\tAl)^*  \tAl}^{-1} } > 0 $ by Lemma \ref{le:Hzlowup}, and (ii)
comes from the following three bounds:
\begin{eqnarray*}
\max_{1\le l\le m}~ \norm{\vu_l}  &\le  &  \max_{1\le l\le m}~ \frac2 {\sqrt{m}} \norm{\va_l \va_l^* \tzl_k} \\
&\le &\frac2 {\sqrt{m}}\cdot \max_{1\le l\le m}~  \norm{\va_l} \cdot \max_{1\le l\le m}~  \abs{\va_l^* \tzl_k} \\
&\lesssim & \sqrt{\frac{n\log m}{m}} \norm{\tzl_k}
\end{eqnarray*}
holds with probability at least $1-O(m^{-10})-O(me^{-1.5n})$ due to the independence between $\va_l$ and $\vzl_k$, and 
\begin{eqnarray*}
\norm{\vw_l} & = &  \frac1{\sqrt m} \Abs{\abs{\va_l^* \tzl_k}^2- \abs{\va_l^* \vxs}^2} \norm{\vu_l} \\
 & \lesssim &   \frac1{\sqrt m}  \cdot \log m\cdot \sqrt{\frac{n\log m}{m}} \norm{\tzl_k} \\
 &\lesssim&  \frac{\sqrt{n \log^3m}}{m} \norm{\tzl_k}^2,
\end{eqnarray*}
and
\begin{eqnarray*}
\abs{\vu_l^* \vv^{(l)}} &=  &  \frac1{\sqrt m} \Abs{\nj{ \left [ \begin{array}{l}  \va_l^* \tzl_k \vspace{0.5em} \\ \overline{ \va_l^* \tzl_k } \end{array}\right] , \left [ \begin{array}{l}  \va_l^* U(\vzl_k) \vv^{(l)}  \vspace{0.5em} \\ \overline{ \va_l^* U(\vzl_k) \vv^{(l)}} \end{array}\right]  } }\\
&\le & \frac2{\sqrt m} \Abs{\va_l^* \tzl_k} \Abs{  \va_l^* U(\vzl_k) \vv^{(l)}  } \\
& \lesssim &  \frac{\log m }{\sqrt m}\norm{\tzl_k} \norm{\vv^{(l)} } \\
&\lesssim & \frac{\log m }{\sqrt m},
\end{eqnarray*}
where the third line comes from the statistically independence between $\vv^{(l)}, \vzl_k$ and $\va_l$, and the last line comes from the fact \eqref{eq:hypaftc} and the bound
\begin{eqnarray*}
\norm{\vv^{(l)} } & = & \norm{\xkh{ (\tAl)^*  \tAl}^{-1} (\tAl)^*  \Fl(\tzl_k) } \\
& \le&  \norm{ \xkh{ (\tAl)^*  \tAl}^{-1}} \norm{(\tAl)^*  \Fl(\tzl_k)} \\
&=&\norm{ \xkh{ (\tAl)^*  \tAl}^{-1}} \cdot  \sqrt 2  \norm{ \frac1m \sum_{j\neq l} \xkh{\abs{\va_j^* \tzl_k}^2- \abs{\va_j^* \vxs}^2 } \va_j\va_j^* \tzl_k } \\
&\le & \frac{10}{19\norm{\tzl_k}^2} \cdot \sqrt 2 \xkh{ 5\delta+ O\xkh{\sqrt{\frac{\log m}{n}}}  +O\xkh{\sqrt{\frac{n\log^3m}{m}} }} \norm{\tzl_k} \\
&\le & O(1),
\end{eqnarray*}
provided $m\gtrsim n\log^3 m$. Here, we use Lemma \ref{le:Hzlowup} and the similar argument to \eqref{eq:theta3} in the second inequality.

$\hfill\square$

\begin{lemma} \label{le:subper}
For any $\vz_k, \tzl_k$, one has
\[
\min_{O \in \mathcal{O}_{2n-1}}\norms{\left [ \begin{array}{l} U(\vz_k) \vspace{0.4em} \\ \overline{ U(\vz_k) } \end{array}\right] - \left [ \begin{array}{l} U(\tzl_k) \vspace{0.4em} \\ \overline{ U(\tzl_k) } \end{array}\right] O}_{2,\Re}  \le  \frac{ \sqrt 2\norm{\vz-\tzl_k} }{\sqrt{\norm{\vz_k}\norm{\tzl_k}} }.
\]
Here, $\mathcal{O}_{2n-1}$ denotes the set of all $(2n-1)\times (2n-1)$ real orthogonal matrix.
\end{lemma}
\begin{proof}
Write the matrix $U(\vz_k) \in \C^{n\times (2n-1)}$ into the form $U(\vz_k)=U_{\Re}(\vz_k)+i U_{\Im}(\vz_k)$, where $U_{\Re}(\vz_k)$ and $U_{\Im}(\vz_k)$ collect entrywise real and imaginary parts of $U(\vz_k)$, respectively. Define
\[
V(\vz_k):= \left [ \begin{array}{l} U_{\Re}(\vz_k) \vspace{0.4em} \\ U_{\Im}(\vz_k) \end{array}\right] \in \R^{2n\times (2n-1)}.
\]
It is easy to verify $V(\vz_k)$ is an orthonormal matrix. We also define $V(\tzl_k)$ accordingly.
Note that
\begin{eqnarray*}
\min_{O \in \mathcal{O}_{2n-1}}\norms{\left [ \begin{array}{l} U(\vz_k) \vspace{0.4em} \\ \overline{ U(\vz_k) } \end{array}\right] - \left [ \begin{array}{l} U(\tzl_k) \vspace{0.4em} \\ \overline{ U(\tzl_k) } \end{array}\right] O}_{2,\Re}& = & \sqrt 2 \min_{O \in \mathcal{O}_{2n-1}} \norms{U(\vz_k)- U(\tzl_k)O}_{2,\Re} \\
&= & \sqrt 2 \min_{O \in \mathcal{O}_{2n-1}}\norm{V(\vz_k)- V(\tzl_k)O}\\
&=& 2 \sqrt{1-\cos\theta_1},
\end{eqnarray*}
where the last equality comes from Lemma \ref{le:pututwosp}.
Since $i\vz_k$ is the normal vector of the space generated by $U(\vz_k)$, it means that the normal vector of $V(\vz_k)$ is $\va:=[-\Im(\vz_k); \Re(\vz_k)]$. Similarly, the normal vector of $V(\tzl_k)$ is $\vb:=[-\Im(\tzl_k); \Re(\tzl_k)]$. By the law of cosines and using Lemma \ref{le:pututwosp} once again, we have
\[
\cos\theta_1=\frac{\norm{\va}^2+\norm{\vb}^2- \norm{\va-\vb}^2}{2\norm{\va}\norm{\vb} } \ge 1- \frac{ \norm{\va-\vb}^2}{2\norm{\va}\norm{\vb}}=1- \frac{ \norm{\vz_k-\tzl_k}^2}{2\norm{\vz_k}\norm{\tzl_k}}.
\]
This completes the proof.

\end{proof}

\subsection{Proof of Lemma \ref{le:z0x}} \label{sec:z0x}  \hfill\\

Without loss of generality we assume $\norm{\vxs}=1$.
Recall that $\vz_0=\sqrt{\lambda_1(Y)/2}\cdot \tz_0$, where $\lambda_1(Y)$ and $\tz_0 \in \C^n$ are the leading eigenvalue and eigenvector of 
\[
Y=\frac1m \sum_{j=1}^m y_j\va_j\va_j^*.
\]
For any $\delta>0$, it follows from Lemma \ref{le:contra} that, with probability exceeding $1-O(m^{-10})$, it holds $\norm{Y-\E Y} \le \delta/2 $.
Note that $\E Y =I_n+ \vxs(\vxs)^*$.  Applying a variant of Wedin’s sin$\Theta$ theorem \cite[Theorem 2.1]{Dopico}, we obtain
\[
\min_{\phi \in \R}\norm{\tz_0 -\vxs e^{i \phi}} \le \frac{ \sqrt 2 \norm{Y-\E Y} }{\lambda_1(\E Y)-\lambda_2(\E Y)} \le \frac1{\sqrt 2} \delta.
\]
Therefore, using the triangle inequality, we have
\begin{eqnarray*}
\min_{\phi \in \R} \norm{\vz_0 -\vxs  e^{i \phi}} &= & \min_{\phi \in \R} \norm{\sqrt{\lambda_1(Y)/2}\cdot \tz_0-\vxs  e^{i \phi}} \\
&\le & \abs{\sqrt{\lambda_1(Y)/2}-1} \cdot \norm{\tz_0}+ \min_{\phi \in \R} \norm{\tz_0 -\vxs e^{i \phi}}  \\
&\le &  \frac14 \delta+ \frac1{\sqrt 2} \delta \le \delta,
\end{eqnarray*}
where the second inequality comes from
\[
 \Abs{\sqrt{\frac{\lambda_1(Y)}2}-1}  \le  \Abs{\frac{\lambda_1(Y)}2-1}  \le \frac{\norm{Y-\E Y}} 2  \le  \frac14 \delta.
\]
Here, we use the Weyl's inequality in the last inequality.

$\hfill\square$

\section{Appendix B: Auxiliary Lemmas}

%\begin{lemma}
%For any matrix $U, X \in \C^{m \times n}$, we have 
%\[
%\dist^2(U,X) \le \frac1{2(\sqrt2-1) \sigma_r^2(X)} \normf{UU^*-XX^*}^2
%\]
%\end{lemma}

\begin{lemma}[Sherman-Morrison formula]  \cite{golub}
If $A$ is a nonsingular $n\times n$ matrix and $\vv$ is a vector, then
\[
(A+\vv\vv^*)^{-1} =A^{-1} -\frac{A^{-1} \vv \vv^* A^{-1}}{1+\vv^* A^{-1} \vv}.
\]
\end{lemma}

\begin{lemma}\label{le:invranone} \cite[Theorem 2.1]{wedin}  \label{le:psudecom}
For any matrix $A,B \in \C^{m\times n}$, one has the decomposition 
\[
B^\dag -A^\dag = -B^\dag (B-A) A^\dag+ (B^*B)^\dag (B-A)^* P_{N(A^*)}+ P_{N(B)} (B-A)^* (AA^*)^\dag ,
\]
where $P_{X}$ is the orthogonal projection onto the subspace $X$, and $N(A^*)$ and $N(B)$ are the nullspaces of $A^*$ and $B$, respectively.
\end{lemma}

\begin{lemma}\cite[Lemma 33]{turstregion} \label{le:pututwosp}
Consider two linear subspace $\mathcal{U}, \mathcal{V}$ of dimension $k$ in $\R^n$ spanned by orthonormal bases $U$ and $V$, respectively. Suppose $0\le \theta_k \le \theta_{k-1}\le   \cdots \le \theta_1\le \pi/2$ are the principal angles between $\mathcal{U}$ and $ \mathcal{V}$. Then it holds
\begin{itemize}
\item[(i)] $\min_{O\in \mathcal{O}_k} \norm{U-VO} \le \sqrt{2-2\cos\theta_1}$;
\item[(ii)] $\theta_1(\mathcal{U}, \mathcal{V})=\theta_1(\mathcal{U}^\perp, \mathcal{V}^\perp)$. Here, $\mathcal{U}^\perp$ and $\mathcal{V}^\perp$ are the orthogonal complement of $\mathcal{U}$ and $\mathcal{V}$, respectively.
\end{itemize}

\end{lemma}

\begin{lemma} \label{le:8.4}
Let $\vxs \in \C^n$ with $\norm{\vxs}=1$. For any vector $\vz_1, \vz_2 \in \C^n$ satisfy
\[
\max\dkh{ \norm{\vz_1-\vxs},  \norm{\vz_2-\vxs}} \le \gamma  \le 1/4.
\]
Denote 
\[
\alpha_1={\rm argmin}_{ |\alpha|=1} ~ \norm{\vz_1-\alpha \vxs}\quad \mbox{and} \quad \alpha_2={\rm argmin}_{ |\alpha|=1} ~ \norm{\vz_2-\alpha \vxs}.
\]
Then we have
\[
\norm{ \overline{\alpha}_1 \vz_1 -\overline{\alpha}_2 \vz_2} \le 6\norm{\vz_1-\vz_2}.
\]
\end{lemma}
\begin{proof}
Applying the triangle inequality, one has
\begin{eqnarray*}
\norm{\overline{\alpha}_1 \vz_1 -\overline{\alpha}_2 \vz_2} & \le & \norm{\overline{\alpha}_1 \vz_1 -\overline{\alpha}_2 \vz_1} +\norm{\overline{\alpha}_2 \vz_1- \overline{\alpha}_2 \vz_2 } \\
&= & \norm{\vz_1} \abs{\alpha_1-\alpha_2}+ \norm{\vz_1-\vz_2}.
\end{eqnarray*}
Recall that $\norm{\vz_1-\vxs} \le \gamma$ and $\norm{\vxs}=1$. It gives
\[
\frac 34\le 1- \gamma  \le \norm{\vz_1} \le  1+ \gamma \le \frac 54.
\]
Therefore, to prove the result, we only need to show $\abs{\alpha_1-\alpha_2} \le 4 \norm{\vz_1-\vz_2}$. To this end, observe that 
\begin{eqnarray*}
\alpha_1 & = & {\rm argmin}_{ |\alpha|=1} ~ \norm{\vz_1-\alpha \vxs}^2 \\
& =& {\rm argmin}_{ |\alpha|=1} ~ \norm{\vz_1} +\| \vxs\|_2-2 \Re(\alpha \vz_1^* \vxs) \\
&=& {\rm argmax}_{ |\alpha|=1} ~ \Re(\alpha \vz_1^* \vxs) \\
&=& {\rm Phase}((\vxs)^* \vz_1).
\end{eqnarray*}
Similarly, we have $\alpha_2={\rm Phase}((\vxs)^* \vz_2)$. It then gives
\begin{eqnarray*}
\alpha_1-\alpha_2 &=&  {\rm Phase}((\vxs)^* \vz_1)- {\rm Phase}((\vxs)^* \vz_2) \\
&\le & \frac{2\xkh{(\vxs)^* \vz_1-(\vxs)^* \vz_2}}{|(\vxs)^* \vz_1|} \\
&\le & \frac{2\| \vxs\|_2 \norm{\vz_1-\vz_2}}{|(\vxs)^* \vz_1|} \\
&\le & 4\norm{\vz_1-\vz_2},
\end{eqnarray*}
where the first inequality comes from the fact that for any $a,b \in \C$, it holds
\[
{\rm Phase}(a)-{\rm Phase}(b)= \Abs{\frac a{|a|}- \frac b{|b|}} \le \Abs{\frac{a-b}{\abs{a}}} + \abs{b} \Abs{\frac 1{\abs{a}}-\frac1{\abs{b}}} \le \frac{2\abs{a-b}}{\abs{a}},
\]
 the second inequality Cauchy-Schwarz inequality, and the last inequality follows from the fact that $\abs{\vz_1^* \vxs} \ge 1/2$. Indeed, since $\norm{\vz_1-\vxs} \le \gamma$, then
 \[
 \norm{\vz_1}^2 +\| \vxs\|_2^2-2 \Re(\alpha \vz_1^* \vxs)  \le \gamma^2.
 \]
It implies that 
\[
\abs{\vz_1^* \vxs} \ge \frac 12 \cdot \xkh{  \norm{\vz_1}^2 +\| \vxs\|_2^2-\gamma^2 } \ge 1- \gamma \ge  \frac12.
\]
This completes the proof.

\end{proof}

The following lemma is the complex version of Lemma 14 in \cite{ChenFan}.
\begin{lemma} \label{le:contra} \cite{ChenFan}.
Fix $\vx^{\sharp} \in\C^n$.  Suppose that $\va_j \sim 1/\sqrt2\cdot \N(0,I_n)+i /\sqrt2\cdot \N(0,I_n), 1\le j\le m$. It holds with probability at least $1-O(m^{-10})$ that 
\[
\norm{ \frac1m \sum_{j=1}^m \abs{\va_j^* \vxs}^2 \va_j\va_j^* - \vxs (\vxs)^*- \norm{\vxs}^2 I_n } \le c_0\sqrt{\frac{n\log^3m}{m}} \norm{\vxs}^2
\]
provided $m\ge C n\log^3 m$ for some sufficiently large constant $C>0$. Furthermore, for any $c_1>1$, it holds with probability at least $1-O(m^{-10})$, 
\[
\norm{ \frac1m \sum_{j=1}^m \abs{\va_j^* \vz}^2 \va_j\va_j^* - \vz\vz^*- \norm{\vz}^2 I_n } \le c_0\sqrt{\frac{n\log^3m}{m}}  \norm{\vz}^2
\]
and 
\[
\norm{ \frac1m \sum_{j=1}^m \xkh{\va_j^* \vz}^2 \va_j\va_j^\T - 2\vz\vz^\T } \le c_0\sqrt{\frac{n\log^3m}{m}}  \norm{\vz}^2
\]
holds simultaneously for all $\vz \in \C^n$ obeying $\max_{1\le j\le m} ~\abs{\va_j^* \vz} \le c_1  \sqrt{\log m}  \norm{\vz}$. Here,  $c_0>0$ is a universal constant.
\end{lemma}

%The next lemma is slightly difference to Lemma \ref{le:contra}. Specifically, 
When utilizing Lemma \ref{le:contra}, it requires the condition that the vector  $\vz \in \C^n$ satisfies $\max_{1\le j\le m} ~\abs{\va_j^* \vz} \lesssim  \sqrt{\log m}  \norm{\vz}$. However, in certain cases, we can only demonstrate that the vectors of interest satisfy $\max_{1\le j\le m} ~\abs{\va_j^* \vz} \lesssim  \sqrt{\log m} $ and $\norm{\vz} \le \delta$ for a constant $\delta > 0$. Hence, a slightly modified version of Lemma \ref{le:contra} is required, as presented below.

\begin{lemma} \label{le:contra2}
Suppose that $\va_j \sim 1/\sqrt2\cdot \N(0,I_n)+i /\sqrt2\cdot \N(0,I_n), 1\le j\le m$.  For any fixed $\beta>1$ and $\delta>0$, assume that $m\ge C\max\xkh{\delta, \beta^4 n\log m}$ for some universal constant $C>0$. Then with probability at least $1-O(m^{-10})$, 
\begin{equation} \label{eq:le551}
\norm{ \frac1m \sum_{j=1}^m \abs{\va_j^* \vz}^2 \va_j\va_j^* - \vz\vz^*- \norm{\vz}^2 I_n } \le c_0 \beta^2 \sqrt{\frac{n\log m}{m}}
\end{equation}
and 
\begin{equation} \label{eq:le552}
\norm{ \frac1m \sum_{j=1}^m \xkh{\va_j^* \vz}^2 \va_j\va_j^\T - 2\vz\vz^\T } \le c_0\beta^2 \sqrt{\frac{n\log m}{m}}
\end{equation}
holds simultaneously for all $\vz \in \C^n$ obeying 
\begin{subequations} \label{eq:nearind000}
\begin{gather} 
\norm{\vz} \le \delta, \label{eq:nearind01}\\
\max_{1\le j\le m} ~\abs{\va_j^* \vz} \le \beta \label{eq:nearind1}.
\end{gather}
\end{subequations}
Here, $c_0>0$ is a universal constant.
\end{lemma}

\begin{proof}
For any unit vector $\vw \in \C^n$ and any $\vz$ obeying \eqref{eq:nearind000}, we have
\[
\frac1m \sum_{j=1}^m \abs{\va_j^* \vz}^2 \abs{\va_j^* \vw}^2=\frac1m \sum_{j=1}^m \abs{\va_j^* \vz}^2  \abs{\va_j^* \vw}^2 \1_{\abs{\va_j^* \vz} \le 2 \beta }.
\]
In what follows, we shall first establish concentration inequalities for the right hand side term for a given $(\vz,\vw)$, and then establish uniform bounds by the standard covering argument. Notice that 
\[
\normpsi{ \abs{\va_j^* \vz}^2  \abs{\va_j^* \vw}^2 \1_{\abs{\va_j^* \vz} \le 2\beta }} \le 4\beta^2\normpsi{ \abs{\va_j^* \vw}^2} \le 4\beta^2 ,
\]
where $\normpsi{\cdot}$ denotes the sub-exponential norm \cite{Vershynin2018}. This further implies that 
\[
\normpsi{ \abs{\va_j^* \vz}^2  \abs{\va_j^* \vw}^2 \1_{\abs{\va_j^* \vz} \le 2\beta} - \E \abs{\va_j^* \vz}^2  \abs{\va_j^* \vw}^2 \1_{\abs{\va_j^* \vz} \le 2\beta} } \le 8\beta^2.
\]
%Here, the inequality comes from the fact that $\normpsi{X-\E X} \le 2\normpsi{X}$.
Applying the Bernstein's inequality, we obtain that for any $0\le \epsilon <1$,
\[
\PP\xkh{ \Abs{ \frac1m \sum_{j=1}^m (\abs{\va_j^* \vz}^2  \abs{\va_j^* \vw}^2 \1_{\abs{\va_j^* \vz} \le 2\beta} -\E \abs{\va_j^* \vz}^2  \abs{\va_j^* \vw}^2 \1_{\abs{\va_j^* \vz} \le 2\beta} ) }\ge 8\epsilon \beta^2 } \le 2\exp(-c\epsilon^2 m),
\]
where $c>0$ is some absolute constant. Taking $\epsilon=C_1 \sqrt{\frac{n\log m}{m}}$ for some large enough constant $C_1>0$, we obtain that 
with probability exceeding $1-2\exp(-c C_1^2 n\log m )$, it holds
\begin{equation} \label{eq:conazaw1}
\Abs{ \frac1m \sum_{j=1}^m \xkh{ \abs{\va_j^* \vz}^2  \abs{\va_j^* \vw}^2 \1_{\abs{\va_j^* \vz} \le 2\beta}-\E \abs{\va_j^* \vz}^2  \abs{\va_j^* \vw}^2 \1_{\abs{\va_j^* \vz} \le 2\beta}} } \le 8 \beta^2 C_1 \sqrt{\frac{n\log m}{m}}.
\end{equation}
Next, we intend to show that \eqref{eq:conazaw1} holds uniformly for all unit vectors $\vw \in \C^n$ and all $\vz \in \C^n$ obeying \eqref{eq:nearind01}.
Define $\mathcal{N}_{\vz}$ to be an $\epsilon_1$-net of $\mathcal{B}_{\vz}(\delta):=\dkh{\vz\in \C^n: \norm{\vz} \le \delta}$ and $\mathcal{N}_0$ an $\epsilon_2$-net of the unit sphere $\mathcal{S}_\C^{n-1}$. In view of \cite[Corollary 4.2.13]{Vershynin2018}, we can choose these nets to guarantee that 
\[
\abs{\mathcal{N}_{\vz}} \le \xkh{1+\frac{2\delta}{\epsilon_1}}^{2n} \qquad \mbox{and} \qquad \abs{\mathcal{N}_{0}} \le \xkh{1+\frac{2}{\epsilon_2}}^{2n}.
\]
For any $\vz \in \C^n$ obeying \eqref{eq:nearind01} and \eqref{eq:nearind1} and any $\vw \in \mathcal{S}_\C^{n-1}$, there exist $\vz_0 \in \mathcal{N}_{\vz}$ and $\vw_0 \in \mathcal{N}_0$ satisfying $\norm{\vz-\vz_0}\le \epsilon_1$ and $\norm{\vw-\vw_0}\le \epsilon_2$. Using the triangle inequality, we have
\begin{eqnarray*}
&& \Abs{ \frac1m \sum_{j=1}^m \abs{\va_j^* \vz}^2  \abs{\va_j^* \vw}^2 \1_{\abs{\va^* \vz} \le 2\beta}- \xkh{\abs{\vz^*\vw}^2+\norm{\vz}^2\norm{\vw}^2 }} \\
&\le &  \underbrace{\Abs{ \frac1m \sum_{j=1}^m \xkh{ \abs{\va_j^* \vz_0}^2  \abs{\va_j^* \vw_0}^2 \1_{\abs{\va_j^* \vz_0} \le 2\beta}-  \E \abs{\va_j^* \vz_0}^2  \abs{\va_j^* \vw_0}^2 \1_{\abs{\va_j^* \vz_0} \le 2\beta}} }}_{:=I_1} \\
&& + \underbrace{\Abs{\frac1m \sum_{j=1}^m \abs{\va_j^* \vz}^2  \abs{\va_j^* \vw}^2 \1_{\abs{\va_j^* \vz} \le 2\beta}- \frac1m \sum_{j=1}^m \abs{\va_j^* \vz_0}^2  \abs{\va_j^* \vw_0}^2 \1_{\abs{\va_j^* \vz_0} \le 2\beta} } }_{:=I_2} \\
&&+ \underbrace{\Abs{   \frac1m \sum_{j=1}^m \E \abs{\va_j^* \vz_0}^2  \abs{\va_j^* \vw_0}^2 \1_{\abs{\va_j^* \vz_0} \le 2\beta} -  \xkh{\abs{\vz_0^*\vw_0}^2+\norm{\vz_0}^2\norm{\vw_0}^2 } } }_{:=I_3}\\
&&+\underbrace{\Abs{ \xkh{\abs{\vz^*\vw}^2+\norm{\vz}^2\norm{\vw}^2 }-\xkh{\abs{\vz_0^*\vw_0}^2+\norm{\vz_0}^2\norm{\vw_0}^2 }}}_{:=I_4}.
\end{eqnarray*}
For the term $I_1$, it follows from \eqref{eq:conazaw1} and the union bound that with probability at least $1-  \xkh{1+\frac{2\delta}{\epsilon_1}}^{2n}  \xkh{1+\frac{2}{\epsilon_2}}^{2n} \cdot 2\exp(-c C_1^2 n\log m )$ that we have
\[
I_1 \le 8 C_1 \beta^2 \sqrt{\frac{n\log m}{m}}.
\]
For the second term $I_2$,  taking $\epsilon_1=1 /m^2$, we obtain that with probability at least $1-O(me^{-1.5m}) $ it holds
\[
\max_{1\le j\le m}\abs{\va_j^*(\vz-\vz_0)}  \le \max_{1\le j\le m} \norm{\va_j} \cdot \norm{\vz-\vz_0} \le \sqrt{6m} \epsilon_1\le \beta,
\]
provided $m\ge 6 \beta^{-1}$. Here, we use the fact that when $m\ge n$ it holds, with probability exceeding $1-O(m e^{-1.5 m})$,  it holds $\max_{1\le j\le m} \norm{\va_j} \le \sqrt {6m}$ in the first inequality.  Therefore, 
\[
\max_{1\le j\le m}\abs{\va_j^* \vz_0} \le \max_{1\le j\le m}\abs{\va_j^*(\vz-\vz_0)} +\max_{1\le j\le m}\abs{\va_j^*\vz} \le 2\beta,
\]
where the last inequality comes from the fact $\vz$ obeying \eqref{eq:nearind1}.
As a result, with probability at least $1-O(m^{-10})$, one has the following identity
\begin{equation} \label{eq:idevent1}
\1_{ \abs{\va_j^* \vz} \le 2\beta } = \1_{ \abs{\va_j^* \vz_0} \le 2\beta } =1.
\end{equation}
Applying the triangle inequality, one has
\begin{eqnarray*}
I_2 &= &\Abs{ \frac1m \sum_{j=1}^m  \xkh{ \abs{\va_j^* \vz}^2  \abs{\va_j^* \vw}^2 - \abs{\va_j^* \vz_0}^2  \abs{\va_j^* \vw_0}^2} } \\
&\le & \Abs{ \frac1m \sum_{j=1}^m  \xkh{ \abs{\va_j^* \vz}^2  - \abs{\va_j^* \vz_0}^2}  \abs{\va_j^* \vw}^2}  +\Abs{ \frac1m \sum_{j=1}^m  \abs{\va_j^* \vz_0}^2  \xkh{\abs{\va_j^* \vw}^2 -   \abs{\va_j^* \vw_0}^2} } \\
&\le & \max_{1\le j\le m}~ \Abs{ \abs{\va_j^* \vz}^2  - \abs{\va_j^* \vz_0}^2 } \cdot  \frac1m \sum_{j=1}^m \abs{\va_j^* \vw}^2+ \max_{1\le j\le m}~ \Abs{ \abs{\va_j^* \vw}^2  - \abs{\va_j^* \vw_0}^2 } \cdot  \frac1m \sum_{j=1}^m \abs{\va_j^* \vz_0}^2 \\
&\overset{\text{(i)}}{\le}  & 2\xkh{\abs{\va_j^* \vz}+ \abs{\va_j^* \vz_0}} \Abs{\va_j^*(\vz-\vz_0)} +2\delta^2 \xkh{\abs{\va_j^* \vw}+ \abs{\va_j^* \vw_0}} \Abs{\va_j^*(\vw-\vw_0)}  \\
& \le & \max_{1\le j\le m}~  \norm{\va_j}^2 \cdot  \xkh{2\xkh{\norm{\vz}+\norm{\vz_0}} \norm{\vz-\vz_0} + 2\delta^2 \xkh{\norm{\vw}+\norm{\vw_0}} \norm{\vw-\vw_0}    }  \\
&\overset{\text{(ii)}}{\le}  & 24 \delta m \xkh{\epsilon_1+\delta  \epsilon_2} \\
& \le &  \frac{24\delta}m.
\end{eqnarray*}
Here, (i) arises from the fact that $ \frac1m \sum_{j=1}^m \va_j\va_j^* \le 2$ with probability exceeding $1-2\exp(-c' n)$ and the facts $\norm{\vw}=1, \norm{\vz_0} \le \delta$,  (ii) comes from the fact $\max_{1\le j\le m} \norm{\va_j} \le \sqrt {6m}$ with high probability, and  (iii) follows from by taking $\epsilon=m^{-1}, \epsilon_2=\delta^{-1}m^{-2}$.

For the term $I_3$, recall the identity \eqref{eq:idevent1}, one has
\[
I_3=  \frac1m \sum_{j=1}^m \E \abs{\va_j^* \vz_0}^2  \abs{\va_j^* \vw_0}^2 \1_{\abs{\va_j^* \vz_0} > 2\beta}=0.
\]

Finally, for the term $I_4$, one can apply the triangle inequality to reach
\begin{eqnarray*}
I_4 &\le&  \Abs{\abs{\vz^*\vw}^2-\abs{\vz_0^*\vw}^2 } + \Abs{\abs{\vz_0^*\vw}^2-\abs{\vz_0^*\vw_0}^2 } + \norm{\vw}^2 \Abs{\norm{\vz}^2-\norm{\vz_0}^2}+\norm{\vz_0}^2 \Abs{\norm{\vw}^2-\norm{\vw_0}^2} \\
&\le & 4\delta \xkh{ \epsilon_1+ \delta \epsilon_2} \le \frac{8\delta}{m^2}.
\end{eqnarray*}

Putting all together, we have 
\begin{eqnarray*}
\Abs{ \frac1m \sum_{j=1}^m \abs{\va_j^* \vz}^2  \abs{\va_j^* \vw}^2 \1_{\abs{\va^* \vz} \le 2\beta}- \xkh{\abs{\vz^*\vw}^2+\norm{\vz}^2\norm{\vw}^2 }} &\le&  8 C_1 \beta^2 \sqrt{\frac{n\log m}{m}}+ \frac{32\delta }{m^2} \\
&\le & c_0 \beta^2 \sqrt{\frac{n\log m}{m}}
\end{eqnarray*}
holds with probability at least 
\[
1-  \xkh{1+\frac{2\delta}{\epsilon_1}}^{2n}  \xkh{1+\frac{2}{\epsilon_2}}^{2n} \cdot 2\exp(-c C_1^2 n\log m )-m\exp(-c' m) \ge 1- O(m^{-10}),
\]
provided $m\ge C\max\xkh{\delta, \beta^4 n\log m}$. This completes the proof of \eqref{eq:le551}.

The proof of \eqref{eq:le552} is similar. The only difference is that the random matrix is not Hermitian, and we need to work with
\[
\Abs{  \vu\xkh{ \frac1m \sum_{j=1}^m \xkh{\va_j^* \vz}^2 \va_j\va_j^\T - 2\vz\vz^\T} \vv } 
\]
for unit vectors $\vu, \vv \in \C^n$. So, we omit it.
\end{proof}

\end{document}